\documentclass[11pt]{article}
\usepackage[utf8]{inputenc}
\usepackage[a4paper]{geometry}

\usepackage{graphicx}
\usepackage{hyperref}
\usepackage{caption}
\usepackage[backend=biber]{biblatex}
\usepackage[title]{appendix}

\usepackage{amsmath}
\usepackage{amssymb}
\usepackage{amsthm}
\usepackage{amsfonts}
\usepackage{mathrsfs}
\usepackage{tikz-cd}

\theoremstyle{plain} \newtheorem{theorem}{Theorem}[section]
\theoremstyle{definition} \newtheorem{definition}[theorem]{Definition}
\theoremstyle{plain} \newtheorem{proposition}[theorem]{Proposition}
\theoremstyle{plain} \newtheorem{lemma}[theorem]{Lemma}
\theoremstyle{plain} \newtheorem{corollary}[theorem]{Corollary}
\theoremstyle{remark} \newtheorem{remark}[theorem]{Remark}
\theoremstyle{remark} 
\theoremstyle{definition} 
\theoremstyle{plain} 
\theoremstyle{definition} 
\theoremstyle{definition} 

\newcommand{\such}{: \,\,}
\DeclareMathOperator{\leqlex}{\preceq_\text{lex}}
\DeclareMathOperator{\llex}{\prec_\text{lex}}
\DeclareMathOperator{\Ap}{Ap}

\mathcode`l="8000
\begingroup
\makeatletter
\lccode`\~=`\l
\DeclareMathSymbol{\lsb@l}{\mathalpha}{letters}{`l}
\lowercase{\gdef~{\ifnum\the\mathgroup=\m@ne \ell \else \lsb@l \fi}}%
\endgroup
\title{The Clifford defect of a numerical semigroup}
\author{
    Eduardo Camps-Moreno\\
    Adrián Fidalgo-Díaz\\
    Umberto Martínez-Peñas\\
    Gretchen L. Matthews
}
\date{}

\begin{document}

\maketitle

\begin{abstract}
    The Clifford defect is a rational number associated to the Weierstrass semigroup at a given point of an algebraic curve. It describes the error-correcting capability of the so-called Modified Algorithm for decoding the corresponding one-point codes defined at the point. This defect also finds applications in other contexts involving one-point codes. We study the Clifford defect of some numerical semigroups arising from curves and give explicit formulas for them.
\end{abstract}

%--------------------------------------------------
\section{Introduction} \label{section:introduction}

Algebraic geometry codes are linear codes defined as the image of the evaluation set of certain algebraic functions with a prescribed pole-set in the points of an algebraic curve. More formally, given  an algebraic nonsingular absolutely irreducible curve $\mathcal{X}$, $n$ $\mathbb{F}_q$-rational points $P_1, \ldots, P_n$, and a divisor $G$ with disjoint support from that of $D := P_1 + \ldots + P_n$, the evaluation code $C_\mathscr{L}(D,G)$ is defined as
\begin{equation*}
    C_\mathscr{L}(D,G) := \{(f(P_1), f(P_2), \ldots, f(P_n)) \in \mathbb{F}_q^n \such f \in \mathscr{L}(G)\},
\end{equation*}
where $\mathscr{L}(G)$ is the Riemann-Roch space associated to $G$. Let $l(G) := \dim \mathscr{L}(G)$. Since the kernel of the evaluation map $f \mapsto (f(P_1), \ldots, f(P_n))$ is exactly $\mathscr{L}(G - D)$, if $\deg G < n$, the evaluation map is injective and the dimension $k$ of $C_\mathscr{L}(D,G)$ is $l(G)$. From the Riemann-Roch Theorem, we get
\begin{equation} \label{equation:bound_l}
    l(G) \geq \deg G - g + 1,
\end{equation}
where $g$ denotes the genus of $\mathcal{X}$. For a similar reason, the minimum distance $d$ is lower-bounded by $n - \deg G$. This implies $d \geq n - k + 1 - g$, a kind of ``lower Singleton bound'' for algebraic geometry codes which depends on $g$. Since the length of these codes is the number of rational points of $\mathcal{X}$ we use to form the divisor $D$, we can interpret the previous inequality as saying that finding curves with many rational points and low genus potentially produces codes with large length and large minimum distance. In fact, algebraic geometry codes were the first example of codes beating the Gilbert-Varshamov bound, showing that there exists a sequence of codes that asymptotically performs better than a family of random codes.

Regarding the question of decoding algebraic geometry codes, in \cite{skorobogatov1990decoding} the Modified Algorithm (MA) was introduced. This algorithm corrects errors in the words of the dual code of $C_\mathscr{L}(D,G)$, which is again an algebraic geometry code \cite[Proposition 2.2.10]{stichtenoth}. Given a received word, the MA computes an error locator and employs it for decoding up to half the minimum distance minus a defect in polynomial time. When $G = m Q$ with $Q$ being a $\mathbb{F}_q$-rational point $Q$ of $\mathcal{X}$, these codes are often called one-point codes and this defect, denoted $s(Q)$, depends entirely on the Weierstrass semigroup of $Q$. The Weierstrass semigroup of $\mathcal{X}$ at $Q$ is defined as
\begin{equation*}
    S := \{a \in \mathbb{N} \such \mathscr{L}((a-1)Q) \neq \mathscr{L}(aQ)\};
\end{equation*}
that is, $S$ is the set of all $a \in \mathbb{N}$ such that there exists a function $f \in \mathbb{F}_q(\mathcal{X})$, with a unique pole at $Q$ of order $\nu_Q(f) = -a$. It is known that $S$ is a numerical semigroup whose number of gaps is the genus of $\mathcal{X}$ \cite[Theorem 1.6.8]{stichtenoth}. The defect $s(Q)$ of the MA is referred to in \cite{duursma1993algebraic} as the Clifford defect and it is defined as
\begin{equation*}
    s(Q) := \max \left \{ \frac{a}{2} - l(aQ) + 1 \such a \in \mathbb{N} \right \}.
\end{equation*}
Then, the MA corrects up to $\lceil \frac{d - 1}{2} \rceil - s(Q)$ errors, where $d$ is the minimum distance of the code. The MA was introduced by Skorobogatov Vladut \cite{skorobogatov1990decoding} as a generalization of \cite{justesen1989construction}. Later, Duursma modified it to the Extended Modified Algorithm (EMA) \cite{duursma1993algebraic}.

The problem of computing the number of errors that the MA is capable of correcting amounts to finding $m \in \mathbb{N}$ such that $s(m) := \frac{m}{2} - l(mQ) + 1$ is as large as possible.  In a previous work \cite{fidalgo2024distributed}, the second and the third authors independently considered an equivalent problem in the context of designing algorithms for distributed matrix multiplication based on one-point codes. More precisely, in \cite{fidalgo2024distributed} the authors defined, for a given numerical semigroup $S$, the map
\begin{equation*}
    \begin{split}
        \Delta: [0,c(S)] &\to \mathbb{N}\\
        a &\mapsto a + 2 |\{b \in S \such a \leq b\}|,
    \end{split}
\end{equation*}
where $c(S)$ denotes the conductor of the semigroup, meaning the smallest integer $c$ such that every integer at least $c$ is an element of the semigroup $S$. If $S$ is the Weierstrass semigroup of $Q$, then $l(mQ) = |\{a \in S \such a \leq m\}|$ and we have that
\begin{equation*}
    2s(m) = \Delta(m) - 2|S \cap [0,c(S)-1]|,
\end{equation*}
so finding $m$ such that $\Delta(m)$ is maximized is equivalent to maximizing $s(m)$. Also, the first and fourth authors made use of the Clifford defect in \cite{camps} in the context of fractional decoding of algebraic geometry codes. For more information regarding algebraic geometry codes, we refer the reader to \cite{stichtenoth} and \cite{pellikaan:handbook} for an approach from the perspective of function fields and algebraic curves, respectively. See \cite{couvreur2021algebraic} for a state-of-the-art survey on applications of algebraic geometry codes, where other more recent decoding algorithms can be found.

Another application of the Clifford defect is lower-bounding the dimension of $\mathscr{L}(mQ)$, which is the dimension of $C_\mathscr{L}(D,G)$ when $m \leq n$. In fact, the name ``Clifford defect'' comes from Clifford's Theorem \cite[Theorem 1.6.13]{stichtenoth}, which states that for a given divisor $G$ of degree $\deg G \leq 2g - 2$, the inequality $l(G) \leq \frac{\deg G}{2} + 1$ holds. When $G = mQ$ with $Q$ a $\mathbb{F}_q$-rational point, the Clifford defect gives us the maximum difference in this inequality, and so
\begin{equation} \label{equation:bound_clifford}
    l(mQ) \geq \frac{m}{2} + 1 - s(Q).
\end{equation}
Assuming that both $s(Q)$ and $g$ are known (and therefore both lower bounds (\ref{equation:bound_l}) and (\ref{equation:bound_clifford}) can be used), we see that (\ref{equation:bound_clifford}) is strictly tighter than (\ref{equation:bound_l}) in the case $G = mQ$ if and only if $2 (g - s(Q) -1) \geq m$. Also, the bound (\ref{equation:bound_clifford}) is nontrivial when $\frac{m}{2} \geq s(Q) - 1$. Consequently, (\ref{equation:bound_clifford}) improves (\ref{equation:bound_l}) when $s(Q) - 1 \leq \frac{m}{2} \leq g - s(Q) - 1$. In Proposition \ref{prop:clifford_upper_bound}, we prove that $0 \leq s \leq \frac{g}{2}$.

In this work, we study the Clifford defect of numerical semigroups inspired by these applications involving algebraic geometry codes, abstracting the notion to apply to any numerical semigroup. In Section \ref{section:basic}, we study some basic notions associated to $\sigma$ (an analogue of $s$ for arbitrary numerical semigroups) and we give an upper-bound on where $\sigma$ is maximized when $S$ is symmetric. In Section \ref{section:simple_gap_sequence} we compute explicit formulas for semigroups arising from some Kummer extensions, a quotient of the Hermitian curve and a generalization of the Klein quartic. Section \ref{section:pedersen_sorensen} is dedicated to the Pedersen-Sørensen curve, for which we compute where $\sigma$ is maximized, partially answering a question from \cite{kirfel1996clifford}, and we give an explicit value of the Clifford defect for a particular case of this curve: the Suzuki curve. In Section \ref{section:norm-trace}, we compute explicit values for the Clifford defect for the norm-trace curve.

\section{Basic properties} \label{section:basic}

We begin this section with notation to be used throughout the paper. Then we introduce some foundational properties needed for later results.

\subsection{Preliminaries}

Unless stated otherwise, $S$ denotes a numerical semigroup, meaning a subset of the set $\mathbb{N}$ of nonnegative integers closed under addition that contains $0$ and has finite complement. We define the genus of $S$ as $g(S) := |\mathbb{N} \setminus S|$. The Frobenius number of $S$ is defined as $F(S) := \max \mathbb{N} \setminus S$, the conductor as $c(S) := F(S) + 1$, and the multiplicity as $m(S) := \min S \setminus \{0\}$. If $S$ is clear from the context, we simply write $g$, $F$ and $c$ for the genus, the Frobenius number and the conductor, respectively.

Given $\lambda := (\lambda_1, \lambda_2, \ldots, \lambda_r), \lambda^\prime := (\lambda_1^\prime, \lambda_2^\prime, \ldots, \lambda_r^\prime) \in \mathbb{Z}^r$, we write $\lambda \leqlex \lambda^\prime$ if the leftmost non-zero entry of $\lambda^\prime - \lambda$ is positive, that is $\leqlex$ is the lexicographical order. We write $\llex$ if $\lambda \leqlex \lambda^\prime$ and $\lambda \neq \lambda^\prime$.

For the remainder of the manuscript, given a numerical semigroup $S$, we define\footnote{Observe that $\sigma = s$ following the notation of \cite{duursma1993algebraic}. We use $\sigma$ since $s$ often denotes an element of $S$.}
\begin{equation*}
    \begin{split}
        \sigma: S \cap [0,c(S)] &\to \mathbb{Q}\\
        s &\mapsto \frac{s}{2} - |S \cap [0,s]| + 1,
    \end{split}
\end{equation*}
from which the Clifford defect of the one-point code with pole divisor supported in $Q$ whose Weierstrass semigroup is $S$ can be determined; in Theorem \ref{theorem:clifford}, we prove that $\sigma(s) \geq 0$.

\begin{remark}
    In \cite{duursma1993algebraic}, $\sigma$ (or $s$, following the notation of \cite{duursma1993algebraic}) is defined with its domain to be $[0,c(S)]$. It is easy to see that $\sigma(x)$ is the maximum in $[0,c(S)]$ if and only if $\sigma(x+1)$ is the maximum in $S \cap [0,c(S)]$. In that case, $\sigma(x+1) = \sigma(x) - \frac{1}{2}$. This means that the problem of maximizing $\sigma$ in $[0,c(S)]$ is equivalent to the problem of finding its maximum in $[0,c(S)] \cap S$. We choose restricting $\sigma$ to $[0,c(S)] \cap S$ because we think that it simplifies this dissertation while preserving the generality. Also, as pointed out in Section \ref{section:introduction}, this problem is equivalent to maximizing $\Delta$ as defined in \cite{fidalgo2024distributed}.
\end{remark}

If we wish to emphasize the semigroup $S$, we write $\sigma_S$ instead of $\sigma$. For convenience, we write $l(s) := |S \cap [0,s]|$ for a general numerical semigroup $S$, for coherence with the case when $S$ is a Weierstrass semigroup. We have two objectives: finding $s \in S \cap [0, c(S)]$ such that $\sigma(s)$ is maximum and computing the explicit value of $\sigma(s)$. We can make the following observation which will simplify computations.

\begin{lemma} \label{lemma:elements_in_between}
    Let $s_1, s_2 \in S$ with $s_1 \leq s_2$. Then $\sigma(s_1) \leq \sigma(s_2)$ if and only if $$|S \cap [s_1 + 1, s_2]| \leq \frac{s_2 - s_1}{2},$$ with equality occurring if and only if $\sigma(s_1) = \sigma(s_2)$.
\end{lemma}

\begin{proof}
    The results follow straightforwardly by definition since
    \begin{equation*}
        \sigma(s_2) - \sigma(s_1) = \frac{s_2 - s_1}{2} + l(s_2) - l(s_1) = \frac{s_2 - s_1}{2} - |S \cap [s_1 + 1, s_2]|.
    \end{equation*}
\end{proof}

\begin{remark} \label{remark:s_minus_1}
    If $s, s-1 \in S$, then $|S \cap [s,s]| = 1 > \frac{1}{2}$, so $\sigma(s-1) > \sigma(s)$ by Lemma \ref{lemma:elements_in_between}. In particular, if $\sigma$ in maximized at $s \in S$, then $s - 1 \notin S$.
\end{remark}

For a general semigroup, $$\sigma(c(S)) = g(S) - \frac{c(S)}{2}.$$ See also \cite[Section 5]{fidalgo2024distributed}.

For the so-called sparse semigroups (semigroups having no consecutive two elements in $[0,c(S)]$), Lemma \ref{lemma:elements_in_between} yields that the map $\sigma$ is maximized at $c(S)$. Examples of sparse semigroups include the ones arising at some points of:
\begin{enumerate}
    \item Elliptic and hyperelliptic curves of genus $g$: $S = \langle 2, 2g + 1 \rangle$ so the Clifford defect is $\sigma(2g+1) = \frac{1}{2}$.
    \item The asymptotically optimal tower of function fields of García-Stichtenoth: see \cite{garcia-stichtenoth-semigroup}.
\end{enumerate}

We consider the map $\phi$ defined by
\begin{equation*}
\begin{split}
    \phi: [0, F(S)] & \to [0, F(S)] \\
             x         & \mapsto F(S) - x,
\end{split}
\end{equation*}
which will be useful for showing the common nature of some proofs in this manuscript. We will use Remark \ref{remark:phi} several times.

\begin{remark} \label{remark:phi}
    From the definition, if $s \in S \cap [0, c(S)]$ then $\phi(s) \notin S$. Consequently, we have that for every $x, y \in \mathbb{N}$, $$x \leq y \leq c(S) \Rightarrow |S \cap [\phi(y), \phi(x)]| \leq y-x+1-|S \cap [x,y]|.$$
\end{remark}

The following result, which is \cite[Prop. 2]{fidalgo2024distributed}, allows us to study the whole range of values $S \cap [0,c(S)]$ by just considering half of them.

\begin{proposition}[{\cite[Prop. 2]{fidalgo2024distributed}} \label{prop:half_clifford}]
    The map $\sigma$ attains its maximum at some $s \in S$ such that $\frac{c(S)}{2} \leq s$.
\end{proposition}

Next, we work toward a proof of Clifford's Theorem for arbitrary numerical semigroups, a result recorded in Theorem \ref{theorem:clifford}. Theorem \ref{theorem:clifford} is an analog of \cite[Theorem 1.6.13]{stichtenoth} which applies to the Weierstrass semigroup of a point on some curve. More precisely, if $S$ is the Weierstrass semigroup at the point $P$, Theorem \ref{theorem:clifford} recovers the classical Clifford's Theorem for the divisors of the form $kP$ with $k \in \mathbb{N}$. We note that for general divisors this is not the case since considering the curve is needed, not only its Weierstrass semigroup. In order to prove Theorem \ref{theorem:clifford}, we need some technical results.

\begin{lemma} \label{lemma:interval}
    Let $x, y, n \in \mathbb{N}$. If $y \in S$ or $n \in S$, then
    \begin{equation*}
        |S \cap [x, x+n-1]| \leq |S \cap ([x, x+n-1] + y)|.
    \end{equation*}
\end{lemma}

\begin{proof}
    The assertion is trivial if $y \in S$. Suppose that $n \in S$. Let us define an injection from $S \cap [x,x+n-1]$ to $S \cap [x+y, x+y+n-1]$. Given $a \in S \cap [x,x+n-1]$, there exists a unique $b \in [x+y, x+y+n-1]$ such that $a \equiv b \mod n$. Since $b = rn + a$ for some $r \in \mathbb{N}$, then $b \in S$. This map gives a different element $b \in S\cap [x+y, x+y+n-1]$ for each distinct $a \in S\cap[x,x+n-1]$ demonstrating the injectivity of the map and the desired result.
\end{proof}

\begin{proposition} \label{prop:low_s}
    Let $s \in S$. If $n \in \mathbb{N}$ is such that $2s + n \leq F(S) - 1$, then $\sigma(s) \leq \sigma(s + n)$.
\end{proposition}

\begin{proof}
    By Lemma \ref{lemma:elements_in_between}, the assertion is equivalent to $|S \cap [s+1, s+n]| \leq \frac{n}{2}$. Set
    \begin{equation*}
        \begin{split}
            X &:= S \cap [s+1, s+n],\\
            Y &:= S \cap [\phi(s + n), \phi(s + 1)].
        \end{split}
    \end{equation*}
    Using Remark \ref{remark:phi}, we obtain that $|X| + |Y| \leq n$. We also have that $|X| \leq |Y|$ because of Lemma \ref{lemma:interval}, so $|X| \leq \frac{n}{2}$ follows.
\end{proof}

\begin{theorem}[Clifford's Theorem] \label{theorem:clifford}
    Let $S$ be a numerical semigroup. Then $l(x) \leq \frac{x}{2} + 1$ for every $x \in \mathbb{N}$ such that $x \leq 2g(S)-2$.
\end{theorem}

\begin{proof}
    If $F(S) \leq x \leq 2g(S)-2$, then
    \begin{equation*}
        l(x) = x - g(S) + 1 < x - \frac{x}{2} - 1 + 1 = \frac{x}{2}.
    \end{equation*}
    If $x \leq F(S)-1$, we can apply Proposition \ref{prop:low_s} to conclude that
    \begin{equation*}
        0 = \sigma(0) \leq \sigma(x) = \frac{x}{2} - l(x) + 1.
    \end{equation*}
\end{proof}

\begin{proposition} \label{prop:clifford_upper_bound}
    Let $S$ be a numerical semigroup. Then $0 \leq \sigma(s) \leq \frac{g(S)}{2}$ for every $s \in S \cap [0,c(S)]$.
\end{proposition}

\begin{proof}
    The inequality $0 \leq \sigma(s)$ follows as a particular case of Theorem \ref{theorem:clifford}. For the remaining inequality, given $s \in S \cap [0,c(S)]$, we have that $s - l(s) = |\{x \in \mathbb{N} \setminus S \such x \leq s\}| - 1$. Thus
    \begin{equation*}
        2 \sigma(s) = |\{ x \in \mathbb{N} \setminus S \such x \leq s \}| - l(s) + 1 \leq |\{x \in \mathbb{N} \setminus S \such x \leq s\}| \leq g(S).
    \end{equation*}
\end{proof}

\subsection{Symmetric semigroups}

Symmetric semigroups, introduced in \cite{Herzog1970}, have long been studied due to their nice properties and connections to applications in commutative algebra and algebraic geometry.
A large number of the Weierstrass semigroups of curves used in coding theory are symmetric; see, for example, the ones in Sections \ref{section:pedersen_sorensen} and \ref{section:norm-trace}. Recall that symmetric semigroups are those for which $c(S) = 2g(S)$ or, equivalently, satisfying that $s \in S$ if and only if $\phi(s) \notin S$ for every $s \in [0, F(S)]$. For symmetric semigroups, an analog of the Riemann-Roch Theorem \cite[Theorem 1.5.15]{stichtenoth} holds.

\begin{theorem}[Riemann-Roch Theorem] \label{theorem:Riemann-Roch}
    For a numerical semigroup $S$ with Frobenius number $F$ and genus $g$, the following are equivalent:
    \begin{enumerate}
        \item $S$ is symmetric.
        \item $l(0) = l(F-1) - g + 1$.
        \item For every $x \in \mathbb{N}$ we have that $l(x) = x + l(F-1-x) - g + 1$.
    \end{enumerate}
     If any of them holds, then $l(s) = s + l(F - s) - g + 1$ for $s \in S$.
\end{theorem}

\begin{proof}
    The equivalence between $1$ and $2$ is clear from the definitions. Also, that $3$ implies $2$ is trivial. For proving that $1$ implies $3$, notice that $y \in S$ if and only if $\phi(y) \notin S$. If $x \notin S$, then
    \begin{equation*}
    \begin{split}
        g &= |\{y \notin S \such y \leq x\}| + |\{y \notin S \such x < y\}|\\
        &= x + 1 - l(x) + |\{y \in S \such y < \phi(x)\}|\\
        &= x + 1 - l(x) + l(\phi(x)-1).
    \end{split}
    \end{equation*}
    Similarly, we obtain the equality when $x \in S$. Observe that, in that case, $l(\phi(x)-1) = l(\phi(x))$.
\end{proof}

\begin{remark}
    If $S$ is symmetric and is the Weierstrass semigroup of the curve $\mathcal{X}$ at the point $P$, then the divisor $(F - 1) P$ is a canonical divisor due to \cite[Prop. 1.6.2]{stichtenoth}.
\end{remark}

\begin{corollary} \label{corollary:Riemann-Roch}
    If $S$ is symmetric, then $\sigma(s) = \sigma(\phi(s)) - \frac{1}{2}$ for every $s \in S$.
\end{corollary}

\begin{proof}
     From Theorem \ref{theorem:Riemann-Roch}, it is straighforward to see that
    \begin{equation*}
        \begin{split}
            \sigma(\phi(s)) &= \frac{\phi(s)}{2} - l(\phi(s)) + 1 = \frac{2g - 1 - s}{2} - l(s) + s - g + 2 = \sigma(s) + \frac{1}{2}.
        \end{split}
    \end{equation*}
\end{proof}

For symmetric semigroups, we can provide another bound on where the Clifford defect is attained, different from the general one of Proposition \ref{prop:half_clifford}. This bound will allow us to study $\sigma$ in the symmetric case in an easier way.

\begin{theorem} \label{theorem:symmetric_bound}
    If $S$ is symmetric, then $\sigma$ attains its maximum at some $s \in S$ such that $g(S) - \left \lceil \frac{m(S)}{2} \right \rceil \leq s \leq g(S)$. Moreover, if $\sigma$ attains its maximum at $s \in |S \cap [0,c(S)-1]|$, then $\sigma(s) = \sigma(\phi(s) + 1)$.
\end{theorem}

\begin{proof}
    Because of Proposition \ref{prop:low_s}, if $s \leq g - 1 - \left \lceil \frac{m(S)}{2} \right \rceil$, then $\sigma(s) \leq \sigma(s + m(S))$, and the lower bound follows. Let $s \in S$ with $\sigma(s)$  maximum. By Remark \ref{remark:s_minus_1}, $s - 1 \notin S$, so $\phi(s) + 1 \in S$. From this observation together with Corollary \ref{corollary:Riemann-Roch},
    \begin{equation*}
        \sigma(s) = \sigma(\phi(s)) - \frac{1}{2} = \frac{\phi(s)-1}{2} - l(\phi(s)) = \frac{\phi(s+1)}{2} - l(\phi(s) + 1) + 1 = \sigma(\phi(s) + 1).
    \end{equation*}
    Hence, if $\sigma$ is maximized at $s \geq g+1$, then it is also maximized at $\phi(s) + 1 \leq g$.
\end{proof}

\subsection{Maximal embedding dimension}

Another important family of numerical semigroups are those of maximal embedding dimension.
Given a numerical semigroup $S$, consider the Apery set $\Ap(S) := \{s \in S \such s-m \notin S\}$. Since $(\Ap(S) \setminus \{0\}) \cup \{m\}$ is a system of generators for $S$, we conclude the well-known bound $e(S) \leq m(S)$, where $e(S)$ denotes the size of the minimal generating set of $S$, often referred to as the embedding dimension of $S$. If the equality $e(S) = m(S)$ holds, we say that $S$ has maximal embedding dimension \cite[Chapter 2]{rosales:numericalSemigroups}.

\begin{theorem}[{\cite[Proposition 3.12]{rosales:numericalSemigroups}}] \label{theorem:maximalembedding}
   A numerical semigroup $S$ has maximal embedding dimension if and only if $T := \{s - m(S) \in \mathbb{N} \such s \in S^\ast\} \cup \{0\}$ is a numerical semigroup.
\end{theorem}

For the following corollary, denote $\sigma_S$ the ``Clifford function'' associated to $S$ and $\sigma_T$, the one associated to $T$.

\begin{corollary} \label{corollary:maximalembedding}
    Let $S$ of maximal embedding dimension and $T$ as in Theorem \ref{theorem:maximalembedding}. Then, $\sigma_S$ is maximized at $s \in S$ if and only if $\sigma_T$ is maximized at $s - m(S) \in T$. Moreover, $\sigma_S(s) = \sigma_T(s-m(S)) + \frac{m(S)}{2} - 1$.
\end{corollary}

\begin{proof}
    It is a straightforward computation. If $s \in S$, then
    \begin{equation*}
        l_S(s) = 1 + |\{t \in T \such t \leq s - m(S)\}| = 1 + l_T(s-m(S)).
    \end{equation*}
    Consequently,
    \begin{equation*}
        \sigma_S(s) = \frac{s-m(S)}{2} + \frac{m(S)}{2} - l_T(s-m(S)) - 1 + 1 = \sigma_T(s-m(S)) + \frac{m(S)}{2} - 1.
    \end{equation*}
\end{proof}

We can interpret Corollary \ref{corollary:maximalembedding} as saying that computing the Clifford defect for all semigroups of maximal embedding dimension is as hard as computing it for all numerical semigroups. Nevertheless, this corollary can be useful in obtaining the Clifford defect of $S$ if the Clifford defect of $T$ is known. See Corollary \ref{corollary:klein} as an example of this application.

%--------------------------------------------------
\section{Some semigroups with simple gap sequences} \label{section:simple_gap_sequence}

In this section, we investigate numerical semigroups $S$ that satisfy a very particular condition on their sequences of gaps (and nongaps), which ensures a simple description of the associated Clifford defects. Informally, this condition can be described as $S$ having an increasing number of consecutive nongaps and a decreasing number of consecutive gaps. More precisely, consider a numerical semigroup of the form
\begin{equation} \label{eq:inc_cons_semigp} S = \{0\} \sqcup \left ( \bigsqcup_{i=1}^r [a_i, b_i] \right) \sqcup [c(S), \infty)
\end{equation}
(where $\sqcup$ denotes the disjoint union) such that $a_i - 1 \notin S$ and $b_i + 1 \notin S$. Denote $l_i := b_i - a_i + 1$ and $g_i := a_{i+1} - b_i - 1$ for $i=0, 1, \ldots r$, where we write $a_0 := b_0 := 0$ and $a_{r+1} := c(S)$.

\begin{proposition} \label{prop:monotone_intervals}
    Let $S$ be a numerical semigroup of the form given in (\ref{eq:inc_cons_semigp}). If $l_0 \leq l_1 \leq \ldots \leq l_r$ and $g_0 \geq g_1 \geq \ldots \geq g_r$, then $\sigma$ is maximized at $a_j$, where $j := \min \{i \such l_i \geq g_i\}$, and $\sigma(a_j) = \frac{a_j}{2} - \sum_{i=1}^{j-1} l_i - 1$.
\end{proposition}

\begin{proof}
    Let $a_i$ for some $i=0,1,\ldots,r+1$. We distinguish cases and use Lemma \ref{lemma:elements_in_between}:
    \begin{enumerate}
        \item If $a_i < a_j$, then
            \begin{equation*}
                |S \cap [a_i+1, a_{i+1}]| = l_i < \frac{l_i + g_i}{2} = \frac{a_{i+1} - a_i}{2} \implies \sigma(a_i) < \sigma(a_{i+1}).
            \end{equation*}
        \item If $a_j < a_i$, then
            \begin{equation*}
                |S \cap [a_{i-1}+1, a_i]| = l_{i-1} \geq \frac{l_i + g_i}{2} = \frac{a_i - a_{i-1}}{2} \implies \sigma(a_i) \leq \sigma(a_{i-1}).
            \end{equation*}
    \end{enumerate}
    Since we know from Remark \ref{remark:s_minus_1} that $\sigma$ is maximized at some $a_i$, we conclude that $\sigma$ is maximized at $a_j$ and the formula for $\sigma(a_j)$ is basic counting.
\end{proof}

For the foregoing, recall $\Ap(S) := \{s \in S \such s-m \notin S\}$, often referred to in the literature as the Apéry set with respect to $m(S)$. It is easy to see that $\Ap(S) := \{0, w_1, w_2, \ldots, w_{m-1}\}$, where $w_i$ is the smallest element of $S$ congruent with $i$ modulo $m$. We write instead $\Ap(S) = \{0 < u_1 < u_2 < \ldots < u_{m-1}\}$. When the Apéry set (equivalently the sequence of gaps) has a particular shape, we can use it for computing the Clifford defect of $S$. Also, we use the notation of (\ref{eq:inc_cons_semigp}) and Proposition \ref{prop:monotone_intervals} in the upcoming results of this section.

\begin{proposition} \label{prop:monotone_intervals2}
    Let $S$ such that, given $k \in \mathbb{N}$, we have $[km, (k+1)m-1] \setminus S = [x_k,y_k]$. Then $\sigma$ is maximized at $a_j$, where $j = \min \{ k \such l_k \geq \frac{m}{2} \} = \max \{k \such a_k \geq u_{\left \lceil \frac{m}{2} \right \rceil} \}$, and
    \begin{equation*}
        \sigma(a_j) = \frac{a_j}{2} - \sum_{i=1}^{j-1} \max \{i \such u_i \leq b_i\}.
    \end{equation*}
\end{proposition}

\begin{proof}
    Observe that $[a_k, b_k] = [y_{k-1}+1, x_k-1]$ for $1 \leq k \leq r$ and
    \begin{equation*}
        \begin{split}
            l_k &= |\{u \in \Ap(S) \such u \leq b_k\}| = \max \{i \such u_i \leq b_k\},\\
            g_k &= |\{u \in \Ap(S) \such u > b_k\}| = m - l_k.
        \end{split}
    \end{equation*}
    So we are under the hypothesis of Proposition \ref{prop:monotone_intervals}. The minimum $k$ such that $l_k \geq g_k$ is the minimum $k$ satisfying $u_{\left \lceil \frac{m}{2} \right \rceil} \leq b_k$, that is, the minimum $k$ such that $l_k \geq \frac{m}{2}$. The rest follows from the previous identities.
\end{proof}

We now consider two particular cases of semigroups satisfying the hypothesis of Proposition \ref{prop:monotone_intervals2}. The first family consists of semigroups whose minimal generators are consecutive (see Figure \ref{figure:interval}). They are known to be the Weierstrass semigroups of Kummer extensions at all the totally ramified places but one (see \cite[Theorem 3.4]{quoosCastellanos}) and, when $S = \langle q, q+1 \rangle$, this is the semigroup of the Hermitian curve at its only point at infinity \cite[Lemma 6.4.4]{stichtenoth}. We need the following lemma.

\begin{figure}
    \center
    \includegraphics[scale=0.2]{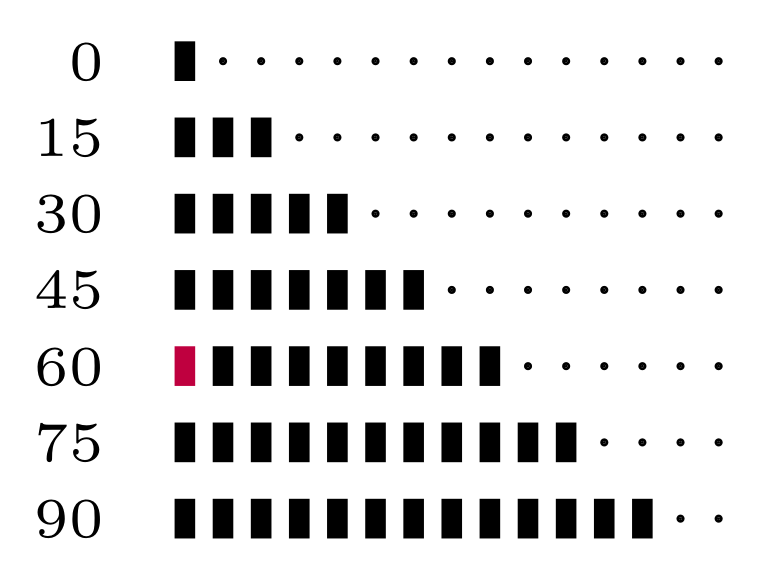}
    \caption{The Clifford defect of $S = \langle 15,16,17 \rangle$ is attained at $45$, shown in red.}
    \label{figure:interval}
\end{figure}

\begin{lemma}[{\cite[Lemma 1]{garciaRosales:interval}}] \label{lemma:interval_parametrization}
    Let $S = \langle m, m+1, \ldots, m + h \rangle$ with $1 \leq h \leq m-1$. Then, $S = \{\lambda_1 m + \lambda_2 \such 0 \leq \lambda_2 \leq h \lambda_1\}$. Moreover given such $\lambda_1, \lambda_2$ and $\lambda_1^\prime, \lambda_2^\prime$, we have $\lambda_1 m + \lambda_2 \leq \lambda_1^\prime m + \lambda_2^\prime$ if and only if $(\lambda_1, \lambda_2) \preceq_\text{lex} (\lambda_1^\prime, \lambda_2^\prime)$, for $\lambda_1^\prime m + \lambda_2^\prime \leq c(S)$.
\end{lemma}

\begin{proof}
    Let $x := \lambda_1 m + \lambda_2$ with $0 \leq \lambda_2 \leq h \lambda_1$ and $\lambda_1, \lambda_2 \in \mathbb{N}$. By Euclidean division, $\lambda_2 = q h + r$ for some $q, r \in \mathbb{N}$ with $0 \leq r < h$. Because of the definition of $\lambda_2$, we have that $q \leq \lambda_1$. If $q = \lambda_1$, then $r=0$ and $x = \lambda_1(m + h) \in S$. If $q \leq \lambda_1-1$, then $x = (\lambda_1 - q - 1) m + q(m + h) + m + r \in S$. Conversely, if $s := \sum_{i=0}^h a_i(m+i) \in S$ for some $a_i \in \mathbb{N}$, then $s = (\sum_{i=0}^h a_i) m + \sum_{i=0}^h a_i i$ and clearly $\sum_{i=0}^h a_i i \leq h \sum_{i=0}^h a_i$.

    For the second part, since $\leqlex$ is total, it suffices to show that $(\lambda_1, \lambda_2) \llex (\lambda_1^\prime, \lambda_2^\prime)$ implies $\lambda_1 m + \lambda_2 < \lambda_1^\prime m + \lambda_2 ^\prime$. From the previous part, we have $c(S) \leq \frac{m-1}{h}$, so it suffices to prove it for $\lambda_1^\prime \leq \frac{m-1}{h}$. If $\lambda_1 = \lambda_1^\prime$ then the result is straightforward. If $\lambda_1 < \lambda_1^\prime$, then
    \begin{equation*}
        \lambda_2 - \lambda_2^\prime \leq h \lambda_1 \leq h \left \lfloor \frac{m - 1}{h} \right \rfloor \leq m - 1 < m \leq h (\lambda_1^\prime - \lambda_1),
    \end{equation*}
    and the assertion holds.
\end{proof}

\begin{corollary} \label{corollary:interval}
    Let $S := \langle m, m + 1, \ldots, m+h \rangle$. Then $\sigma$ is maximized at $\lambda m$, where $\lambda := \left \lceil \frac{m-2}{2h} \right \rceil$, and
    \begin{equation*}
        \sigma \left (\lambda m \right) = \lambda \left (\frac{m}{2} - 1\right) - h \binom{\lambda}{2}.
    \end{equation*}
\end{corollary}

\begin{proof}
    First, observe that $[a_k, b_k] = [km, km + kh]$ for $b_k < c(S)$ from Lemma \ref{lemma:interval_parametrization}. Then we can apply Proposition \ref{prop:monotone_intervals2}. We have $l_k = kh + 1$ and $g_k = m - kh-1$, so $l_k \geq g_k$ if and only if $k \geq \left \lceil \frac{m-2}{2h} \right \rceil$. From Proposition \ref{prop:monotone_intervals}, $\sigma$ is maximized at $\lambda a$.
    For the rest, we have from Lemma \ref{lemma:interval_parametrization},
    \begin{equation*}
        l(\lambda m) = 1 + \sum_{i=0}^{\lambda - 1} ( k i + 1) = 1 + k \binom{\lambda}{2} + \lambda,
    \end{equation*}
    and the result follows.
\end{proof}

The last result can be used to compute the Clifford defect of the Weierstrass semigroup for another curve: the Klein quartic. Defined over $\mathbb{F}_q$ by $x^3 y + y^3 + x=0$ with $\gcd(q,7) = 1$, this is a curve commonly used as an example in the literature of algebraic geometry codes. A generalization defined by $x^m y + y^m + x=0$ with $\gcd(q,m^2 - m + 1) = 1$
is studied in \cite{brasamoros:semigroupsCodes}; its Weierstrass semigroup at a given point is $S := \{i(m-1) + jm \such i,j \in \mathbb{N} \quad j \neq 0\} \cup \{0\}$ (see Figure \ref{figure:klein}). It is easy to see that its minimal set of generators is $\{m, 2m - 1, 3m - 2, \ldots, (m-1)m - m + 1\}$. Let us compute its Clifford defect.

\begin{figure}[h]
    \center
    \includegraphics[scale=0.23]{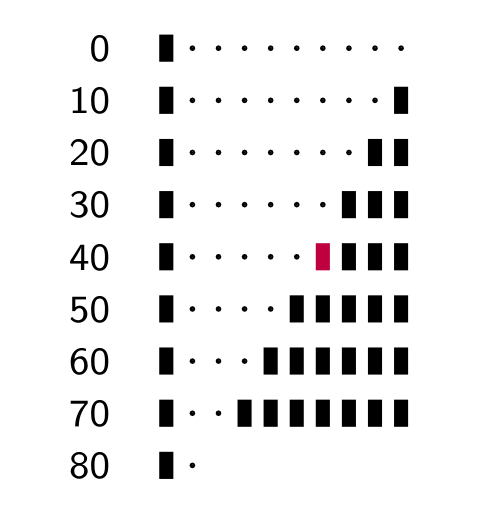}
    \caption{Klein semigroup for $m = 10$. The Clifford defect is attained at $46$, drawn in red.}
    \label{figure:klein}
\end{figure}

\begin{corollary} \label{corollary:klein}
Let $S = \langle m, 2m - 1, 3m -2, \ldots, (m-1)^2 \rangle$. Then $\sigma$ is maximized at $s := \left \lceil \frac{m-1}{2} \right \rceil (m-1) + 1$ and
\begin{equation*}
    \sigma(s) := \frac{1}{2} \left \lceil \frac{m-1}{2} \right \rceil \left ( \left \lfloor \frac{m-1}{2} \right \rfloor + 1 \right ) - \frac{1}{2}.
\end{equation*}
\end{corollary}

\begin{proof}
    We have $S = (m + T) \cup \{0\}$ where $T = \langle m-1, m \rangle$. From Corollary \ref{corollary:interval}, we have $\sigma_T$ is maximized at $\left ( \left \lceil \frac{m-1}{2} \right \rceil - 1 \right)(m-1)$ with a value of $\frac{1}{2} \left \lfloor \frac{m-1}{2} \right \rfloor \left ( \left \lceil \frac{m-1}{2} \right \rceil - 1\right)$. The result follows from Corollary \ref{corollary:maximalembedding}.
\end{proof}

In \cite{schmidt1939arithmetischen}, it was proved that all points of a curve but a finite number of them, called Weierstrass points, define the same Weierstrass semigroup over an algebraically closed field. A curve is said to be classical if almost all points have $\mathbb{N} \setminus \{1,2,3,\ldots,g\}$ as its Weierstrass semigroup and nonclassical otherwise. The first discovered example of a nonclassical curve was $y^q + y = x^m$ where $q +1 = m r$ and $r \in \mathbb{N}$, given in \cite{garcia1986weierstrass}. This curve is a quotient of the Hermitian curve, with a unique point at infinity and the Weierstrass points being all the $\mathbb{F}_{q^2}$-rational points of the curve. Their Weierstrass semigroup is the same for all these points, namely $S := \langle m, q \rangle$ (see Figure \ref{figure:quotient_hermitian}). Since $S$ satisfies the hypothesis of Proposition \ref{prop:monotone_intervals2}, we can compute its Clifford defect. We need the following lemma.

\begin{figure}
    \centering
    \includegraphics[width=0.18\linewidth]{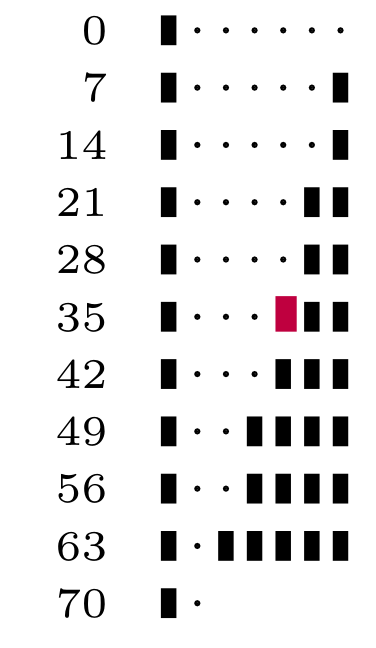}
    \caption{The Clifford defect
    for the semigroup $S = \langle 7,13 \rangle$
    is attained at $39$, shown in red.}
    \label{figure:quotient_hermitian}
\end{figure}

\begin{lemma} \label{lemma:hermitian_quotient}
    Given $S = \langle m, q \rangle$ where $q+1 = a m$ for some $r \in \mathbb{N}$. Then
    \begin{equation*}
        S = \{\lambda_1 m - \lambda_2 \such 0 \leq a \lambda_2 \leq \lambda_1\}.
    \end{equation*}
    Moreover given such $\lambda_1, \lambda_2$ and $\lambda_1^\prime, \lambda_2^\prime$, we have $\lambda_1 m - \lambda_2 \leq \lambda_1^\prime m - \lambda_2^\prime$ if and only if $(\lambda_1, -\lambda_2) \leqlex (\lambda_1^\prime, -\lambda_2^\prime)$, for $\lambda_1^\prime m - \lambda_2^\prime \leq c(S)$.
\end{lemma}

\begin{proof}
    If $0 \leq r \lambda_2 \leq \lambda_1$ as in the statement, then $\lambda_1 m - \lambda_2 = (\lambda_1 - r \lambda_2) m + \lambda_2 q \in S$. Conversely, for a given $a, b \in \mathbb{N}$, we have $am + bq = (a + rb) m - b$ and clearly $0 \leq rb \leq a + rb$.

    For the second part, it suffices to prove that $\lambda_1 m - \lambda_2 < (\lambda_1 + 1) - \lambda_2^\prime$ for $0 \leq r\lambda_2^\prime \leq \lambda_1+1$. We have that $c(S) = (q-1)(m-1) = m(q-r-1)$, so also it suffices to consider $\lambda_1 < q$. Since $ \lambda_2^\prime - \lambda_2 \leq \frac{\lambda_1+1}{r} < \frac{q+1}{r} = m$,
    the result follows.
\end{proof}

\begin{corollary} \label{corollary:hermitian_quotient}
    Let $S = \langle m, q \rangle$ with $q+1 = r m$. Then $\sigma$ is maximized at $s = q \left (\left \lceil \frac{m}{2} \right \rceil - 1\right)$ and
    \begin{equation*}
        \sigma (s) =
        \begin{cases}
            \frac{1}{8}(m-1) (q - 1) &\text{if } m \text{ is even}\\
            \frac{1}{8}(m-1) (q - r - 1) &\text{if } m \text{ is odd}.
        \end{cases}
    \end{equation*}
\end{corollary}

\begin{proof}
    We have that $[a_k, b_k] = \left [km - \left \lfloor \frac{k}{r} \right \rfloor, km \right ]$ from Lemma \ref{lemma:hermitian_quotient}. So $l_k = \left \lfloor \frac{k}{r} \right \rfloor + 1$, and $\sigma$ is maximized at the first $a_k$ such that $l_k = \left \lfloor \frac{k}{r} \right \rfloor \geq \frac{m}{2}$ from Proposition \ref{prop:monotone_intervals2}. This happens when $k = \lambda := r \left (\left \lceil \frac{m}{2} \right \rceil - 1\right )$ and $a_\lambda = \lambda - \left \lceil \frac{m}{2} \right \rceil + 1 = q \left (\left \lceil \frac{m}{2} \right \rceil - 1\right)$. Moreover,
    \begin{equation*}
        l(a_\lambda) = 1 + \sum_{k=0}^{\lambda - 1} \left ( \left \lfloor \frac{k}{r} \right \rfloor + 1\right ) = 1 + \lambda + \sum_{k=0}^{\left \lceil \frac{m}{2} \right \rceil -2}  rk = \lambda + 1 + r \binom{\left \lceil \frac{m}{2} \right \rceil-1}{2} = 1 + r \binom{\left \lceil \frac{m}{2} \right \rceil}{2},
    \end{equation*}
    from which we obtain
    \begin{equation*}
        \sigma(a_\lambda) = \frac{q}{2} \left (\left \lceil \frac{m}{2} \right \rceil - 1 \right) - r \binom{\left \lceil \frac{m}{2} \right \rceil}{2} = \frac{1}{2} \left ( \left \lceil \frac{m}{2} \right \rceil - 1\right) \left (q - r \left \lceil \frac{m}{2} \right \rceil \right).
    \end{equation*}
    The result then follows from straightforward computations.
\end{proof}

%--------------------------------------------------
\section{Pedersen-Sørensen semigroups} \label{section:pedersen_sorensen}

In \cite{pedersenSorensen}, the function field defined by $y^q - y - x^{q_0}(x^q - x) = 0$, where $q := t q_0^2$ for some $t, q_0 \in \mathbb{N}$, was introduced as a generalization of the Suzuki curve \cite{hansen1990group}. It was shown to have genus $g = \frac{q(q-1)}{2q_0}$ and $q^2 + 1$ $\mathbb{F}_q$-rational places with exactly one of them at infinity. In addition, its Weierstrass semigroup at this point was computed to be $S = \langle q, q + q_0, q + tq_0, (t-1)q + t q_0 + 1 \rangle$ (see Figure \ref{figure:pedersen-sorensen}). The semigroup can also can be described by the following sequence of gluings \cite[Subsection 8.3]{rosales:numericalSemigroups}:
\begin{equation*}
    \begin{split}
        S_0 &:= \mathbb{N},\\
        S_1 &:= q_0 S_0 + \langle q_0 + 1 \rangle,\\
        S_2 &:= t S_1 + \langle t q_0 + 1 \rangle,\\
        S &:= q_0 S_2 + \langle q_0((t-1)tq_0 + t) + 1 \rangle.
    \end{split}
\end{equation*}
This kind of numerical semigroup is often referred to in the literature as free numerical semigroups \cite[Subsec. 2.3]{danna:numericalsemigroups} or telescopic \cite{kirfelpellikaanarray}. In particular, telescopic semigroups are symmetric, and so is $S$. In this section, we study the Clifford defect of $S$.
We give an explicit description on where $\sigma$ is maximized and give an explicit formula of this maximum
for the particular case of the Suzuki curve in Subsection \ref{subsection:suzuki},
complementing the work in \cite{kirfel1996clifford} where some bounds were given together with its asymptotical behavior.
For computing where $\sigma$ is maximized, we distinguish cases depending on the parity of $p$, the characteristic of $\mathbb{F}_q$.

\begin{figure}[h]
    \center
    \includegraphics[scale=0.23]{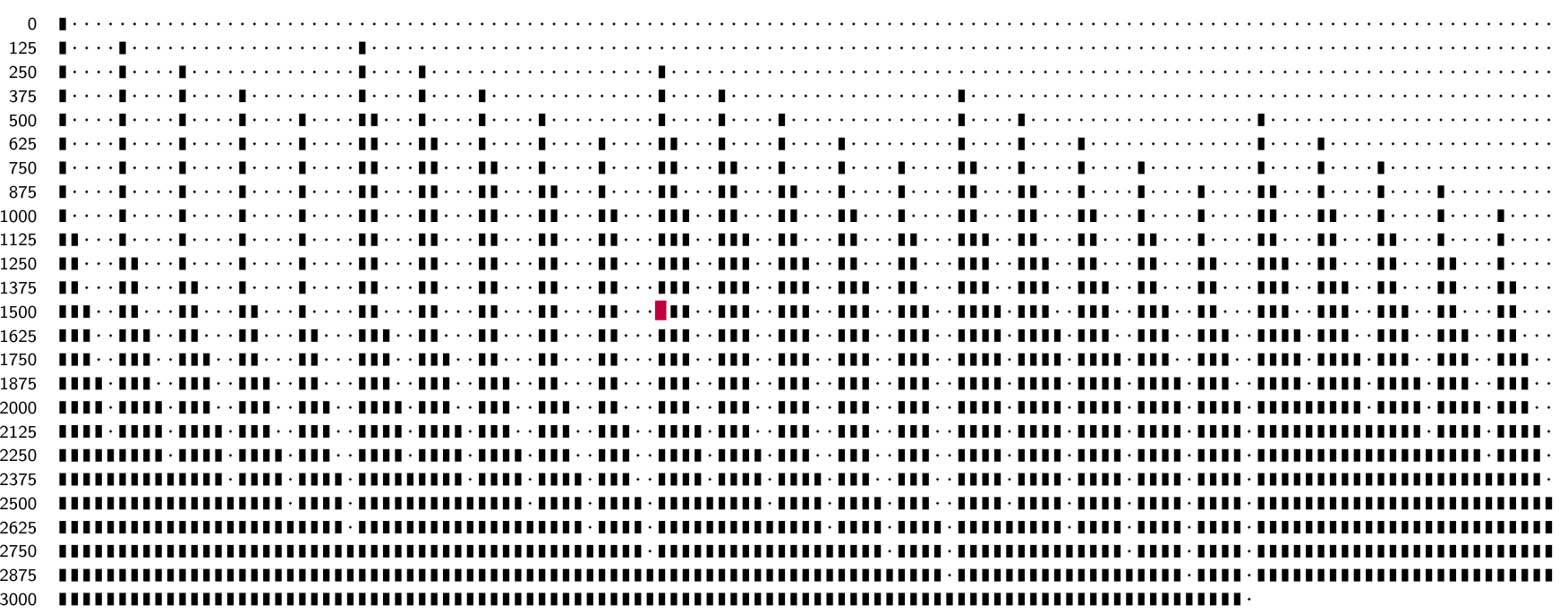}
    \caption{The Clifford defect of the
    Pedersen-Sørensen semigroup for $q_0 = t = 5$
    is attained at $g = 1550$, shown in red.}
    \label{figure:pedersen-sorensen}
\end{figure}

\subsection{The case $p \neq 2$}

In this subsection, we consider the Pedersen-Sørensen semigroup in odd characteristic.

\begin{lemma} \label{lemma:pedersen_pseudoparametrization}
    Consider $a, b, c \in \mathbb{N}$ such that the following conditions hold:
    \begin{enumerate}
        \item $0 \leq b \leq t-1$.
        \item $0 \leq c \leq \frac{q_0 - 1}{2}$.
        \item $a + b + (t-2)c \leq \frac{(t-1)q_0}{2}$.
    \end{enumerate}
    Then $g + atq_0 + bq_0 + c \in S$.
\end{lemma}

\begin{proof}
    First of all, notice that we can write $g$ as $g = \frac{q_0t - 1}{2} q + \frac{q_0 - 1}{2} tq_0$. By a simple computation, we obtain the result since $g + a tq_0 + b q_0 + c=$
    \begin{equation*}
    \begin{array}{ll}
        \left (\frac{q_0 t - 1}{2} - a - b - (t-1)c \right ) q &+ \left (\frac{q_0 - 1}{2} - c \right) tq_0 \\&+ a (q + tq_0) + b (q + q_0) + c \left ( (t-1)q + tq_0 + 1 \right).
        \end{array}
    \end{equation*}
    It follows from condition $3$ that the sum of the first two terms is in the semigroup. The remaining terms are themselves in $S$ too.
\end{proof}

For the rest of the subsection, let us fix some notation. We define $T \subseteq \mathbb{N}^3$ as the set of triples $(a,b,c)$ satisfying the hypothesis of Lemma \ref{lemma:pedersen_pseudoparametrization} and $R \subseteq S$ as the set
\begin{equation*}
    R := \{g + atq_0 + b q_0 + c \such (a,b,c) \in T\}.
\end{equation*}
Notice that the representation of elements of $R$ as triples in $T$ is unique. In other words, $|R| = |T|$. For a given $x \in \mathbb{N}$, we define $l^\prime(x) := |R \cap [g, g + x-1]|$ and we set $\tau(x) := l^\prime(x) - \frac{x}{2}$ for $g \leq x \in S$.

\begin{remark} \label{remark:ps_objective}
    This battery of definitions is motivated by the fact that, since $$|S \cap [g+1, g + x]| - \frac{x}{2} \geq \tau(x)$$ because of Lemma \ref{lemma:pedersen_pseudoparametrization}, it suffices to show that $\tau(x) \geq 0$ in order to conclude that $\sigma(g) \geq \sigma(g + x)$ (Lemma \ref{lemma:elements_in_between}). The objective then is to show that $\tau(x) \geq 0$ for every $0 \leq x \leq \frac{q+1}{2}$ because of Theorem \ref{theorem:symmetric_bound}.
\end{remark}

\begin{lemma} \label{lemma:ps_lower_bound_param}
    Let $(a,b,0) \in T$ with $a \leq \frac{q_0 - 1}{2}$. Then
    \begin{equation*}
        \frac{q_0 - 1}{2} \leq |R \cap [g + a tq_0 + bq_0, g + a tq_0 + (b+1)q_0 - 1]|.
    \end{equation*}
\end{lemma}

\begin{proof}
    We count the number of ways of picking $0 \leq c \leq \frac{q_0 - 1}{2}$ such that $c(t-2) \leq \frac{(t-1)q_0}{2} - a - b$. By taking the extremal values $a = \frac{q_0 - 1}{2}$ and $b = t-1$, we obtain
    \begin{equation*}
        (t-2) c \leq \frac{(t-1)q_0}{2} - \frac{q_0 - 1}{2} - t+1 = \frac{(t-2)(q_0 - 2) - 1}{2} \implies c \leq \frac{q_0 - 1}{2} - \frac{1}{2(t-2)}.
    \end{equation*}
    Thus, the number of possible values for $c$ is $\frac{q_0 - 1}{2}$.
\end{proof}

\begin{lemma} \label{lemma:ps_final_g} For $t$ and $q_0$ as above,
    $\tau(\frac{q_0 + 1}{2} tq_0) \geq 0$.
\end{lemma}

\begin{proof}
    %We have to show that
    %\begin{equation*}
        %|R \cap [g, g + \frac{q_0 + 1}{2} tq_0 - 1]| \geq \frac{(q_0 + 1) t q_0}{4}.
    %\end{equation*}
    The condition can be rewritten as $|R \cap [g, g + \frac{q_0 + 1}{2} tq_0 - 1]| - \frac{(q_0 + 1) t q_0}{4} =$
    \begin{equation*}
         \sum_{k = 0}^{\frac{q_0+1}{2} t - 1} \left (|R \cap [g + kq_0, g + (k+1)q_0 - 1] | - \frac{q_0}{2} \right) \geq 0.
    \end{equation*}

    Let us study the number $E(k) := |R \cap [g + kq_0, g+(k+1)q_0 - 1]| - \frac{q_0}{2}$. From the definition of $R$ and Lemma \ref{lemma:ps_lower_bound_param}, we obtain two possible values for $E(k)$ since $-\frac{1}{2} \leq E(k) \leq \frac{1}{2}$. We consider for which values of $k$, $E(k)=\frac{1}{2}$ and for which, $E(k)=-\frac{1}{2}$.

    Consider $k = a t + b$ with $0 \leq b \leq t-1$ and set $M := \frac{q_0+1}{2}$. Among the $tM$ possible pairs $(a,b)$, we have that $E(k) = \frac{1}{2}$ if
    \begin{equation*}
        a + b \leq \frac{tq_0 - 1}{2} - \frac{q_0 - 1}{2}(t-1) = \frac{t + q_0}{2} - 1.
    \end{equation*}
    If the number of such pairs $(a,b)$ is at least $\frac{tM}{2}$, then we have
    \begin{equation*}
        \sum_{k=0}^{\frac{q_0 + 1}{2} t - 1} E(k) \geq \frac{1}{2} \frac{tM}{2}  - \frac{1}{2} \left (tM - \frac{tM}{2} \right) = 0,
    \end{equation*}
    and the result is proven. The number of such pairs is exactly $N := \sum_{i=0}^{\frac{q_0-1}{2}} \min \{\frac{t+q_0}{2} -i, t\}$. We distinguish two cases:
    \begin{enumerate}
        \item If $t \geq q_0$, then
            \begin{equation*}
                N = \sum_{i=0}^{\frac{q_0-1}{2}} \left (\frac{t+q_0}{2} -i \right) = \frac{q_0 + 1}{2}\left ( \frac{t+q_0}{2} - \frac{q_0 -1}{4} \right) = \frac{tM}{2} + \frac{(q_0 + 1)^2}{8} \geq \frac{tM}{2}.
            \end{equation*}

        \item If $t < q_0$, then
            \begin{equation*}
                \begin{split}
                    N &= \sum_{i=0}^{\frac{q_0 - t}{2} - 1}t + \sum_{i=\frac{q_0-t}{2}}^\frac{q_0 - 1}{2} \left (\frac{t+q_0}{2} - i \right)= \frac{q_0 - t}{2} t + \sum_{i=0}^\frac{t-1}{2} (t - i) = \frac{q_0 + 1}{2}t - \frac{t^2 - 1}{8}\\
                    &\geq tM - \frac{t(q_0 + 1)}{8} \geq \frac{tM}{2}.
                \end{split}
            \end{equation*}
    \end{enumerate}
In each case, we see at least $\frac{tM}{2}$ pairs and the conclusion holds.
\end{proof}

\begin{definition}
    We say that a map $f: [a,b] \subseteq \mathbb{N} \to \mathbb{R}$ is asymptoticallly decreasing if $f(x) \geq f(x+1)$ implies that $f(x+1) \geq f(x+2)$.
\end{definition}

\begin{remark}
    \label{remark:asymptotically_decreasing}
    Observe that, if $f$ is asymptotically decreasing, then its minimum it is attained in $f(a)$ or in $f(b)$. This is the key observation for the following result.
\end{remark}

\begin{theorem} \label{theorem:ps_odd}
    The Clifford defect of the Pedersen-Sørensen semigroup $S$ in odd characteristic is attained at $g$.
\end{theorem}

\begin{proof}
    Because of Theorem \ref{theorem:symmetric_bound}, the Clifford defect is attained at some $g \leq n \leq g + \frac{q+1}{2}$. We proceed in several steps to prove that $\tau(x) \geq 0$ for each suitable $x$ as in Remark \ref{remark:ps_objective}. We use Remark \ref{remark:asymptotically_decreasing} to simplify the domain in which to search for the minimum of $\tau$.

    Consider $x = kq_0 + r$ and the map $\tau_1(r) := \tau(kq_0 + r)$. Because of Lemma \ref{lemma:pedersen_pseudoparametrization}, $\tau_1(r)$ is asymptotically decreasing so it suffices to consider the extremes of the domain of $\tau_1$, so we only have to consider $x = kq_0$.

    Consider $x = atq_0 + bq_0$ and the map $\tau_2(b) := \tau(atq_0 + bq_0)$. As in the previous paragraph, $\tau_2$ is asymptotically decreasing so we only have to consider $x = atq_0$.

    Consider $x = atq_0$ and the map $\tau_3(a) := \tau(atq_0)$. Again, $\tau_3$ is asymptotically decreasing so it suffices to check $\tau(0) \geq 0$ and $\tau(\frac{q_0 + 1}{2}tq_0) \geq 0$. The first is trivially true and the second it is exactly Lemma \ref{lemma:ps_final_g}.
\end{proof}

\subsection{The case $p = 2$} \label{subsection:pedersen_even}

In this subsection, we consider the Pedersen-Sørensen semigroup in even characteristic.

\begin{lemma} \label{lemma:ps_even1}
    Consider $a, b, c \in \mathbb{N}$ such that the following conditions hold:
    \begin{enumerate}
        \item $0 \leq b \leq t-1$.
        \item $0 \leq c \leq \frac{q_0}{2} - 1$.
        \item $a + b + (t-2)c \leq \frac{q_0(t-1)}{2} - 1$.
    \end{enumerate}
    Then $g - \frac{q}{2} + \frac{t}{2}q_0 + atq_0 + bq_0 + c \in S$.
\end{lemma}

\begin{proof}
    First notice that $g - \frac{q}{2} + \frac{t}{2}q_0 = (\frac{tq_0}{2} - 1)q + \frac{q_0}{2}tq_0$. Now, $g - \frac{q}{2} + \frac{t}{2}q_0 + atq_0 + b q_0 + c=$
    \begin{equation*}
        \left (\frac{tq_0}{2} - 1 - a - b - (t-1)c \right) q + \left (\frac{q_0}{2} - c \right) tq_0 + a(q + tq_0) + b(q + q_0) + c((t-1)q + tq_0 + 1).
    \end{equation*}
    The sum of the first two terms belongs to the semigroup due to condition $3$. Each of the remaining summands belongs to $S$ also.
\end{proof}

\begin{lemma} \label{lemma:ps_even2}
    Consider $b, c \in \mathbb{N}$ such that the following conditions hold:
    \begin{enumerate}
        \item $0 \leq b \leq \frac{t}{2} - 1$.
        \item $0 \leq c \leq \frac{q_0}{2} - 1$.
    \end{enumerate}
    Then $g - \frac{q}{2} + bq_0 + c \in S$.
\end{lemma}

\begin{proof}
    As in Lemma \ref{lemma:ps_even1}, notice that $g - \frac{q}{2} = \left (\frac{tq_0}{2} - 1 \right)q + \left (\frac{q_0}{2}-1 \right)tq_0 + \frac{tq_0}{2}$ and we have that $g - \frac{q}{2} + bq_0 + c = $
    \begin{equation*}
        \left (\frac{tq_0}{2} - 1 - b - \frac{t}{2} - (t-1)c \right ) q + \left (\frac{q_0}{2} - 1 - c \right) tq_0 + c((t-1)q + tq_0 + 1) + \left (b+\frac{t}{2} \right)(q + q_0),
    \end{equation*}
    which is in the semigroup.
\end{proof}

\begin{theorem} \label{theorem:ps_even}
    The Clifford defect of the Pedersen-Sørensen semigroup $S$ in even characteristic  is attained at $g - \frac{q}{2}$.
\end{theorem}

\begin{proof}
    By Lemma \ref{lemma:elements_in_between}, it suffices to check that $|S \cap [g - \frac{q}{2} + 1, g - \frac{q}{2} + x]| \geq \frac{x}{2}$ for every $0\leq x \leq \frac{q}{2}$. Using the same argument regarding $\tau_i$ being assymptotically decreasing as in the proof of Theorem \ref{theorem:ps_odd}, we only have to check $x = \frac{q}{2}$. But using Lemmas \ref{lemma:ps_even1} and \ref{lemma:ps_even2} we obtain that
    \begin{equation*}
        \left |S \cap \left [g - \frac{q}{2}, g-1 \right ] \right | = \sum_{i=0}^{\frac{t}{2}q_0 - 1} \left |S \cap \left [g - \frac{q}{2} + iq_0, g - \frac{q}{2} + (i+1)q_0 - 1\right ] \right |
        \geq \sum_{i=0}^{\frac{t}{2}q_0 - 1} \frac{q_0}{2}
        = \frac{q}{2}.
    \end{equation*}
\end{proof}

\subsection{Suzuki semigroup} \label{subsection:suzuki}

The most studied
Pedersen-Sørensen curve is the Suzuki curve, that is when $t=2$. The Suzuki curve was introduced in \cite{hansen1990group} and has been employed for constructing algebraic geometry codes since then. It has even been shown that yields codes attaining the best possible known parameters for some regimes (codes of length $64$ and $65$ over $\mathbb{F}_8$ \cite{Grassl:codetables}). Its Weierstrass semigroup at its only point at infinity is
\begin{equation*}
    S := \langle q, q + q_0, q + 2q_0, q + 2q_0 + 1 \rangle,
\end{equation*}
where $q_0 = 2^h$ and $q = 2q_0^2$ for $h \geq 1$. As pointed out previously in this section, it is a symmetric semigroup with genus $g = q_0(q-1)$ (see Figure \ref{figure:suzuki}).

As a particular case of Subsection \ref{subsection:pedersen_even}, we obtain where $\sigma$ is maximized for $S$. Moreover, we give an explicit formula for this maximum. We proceed by enumerating the semigroup in a way that is compatible with the lexicographical order.

\begin{figure}[h]
    \center
    \includegraphics[scale=0.22]{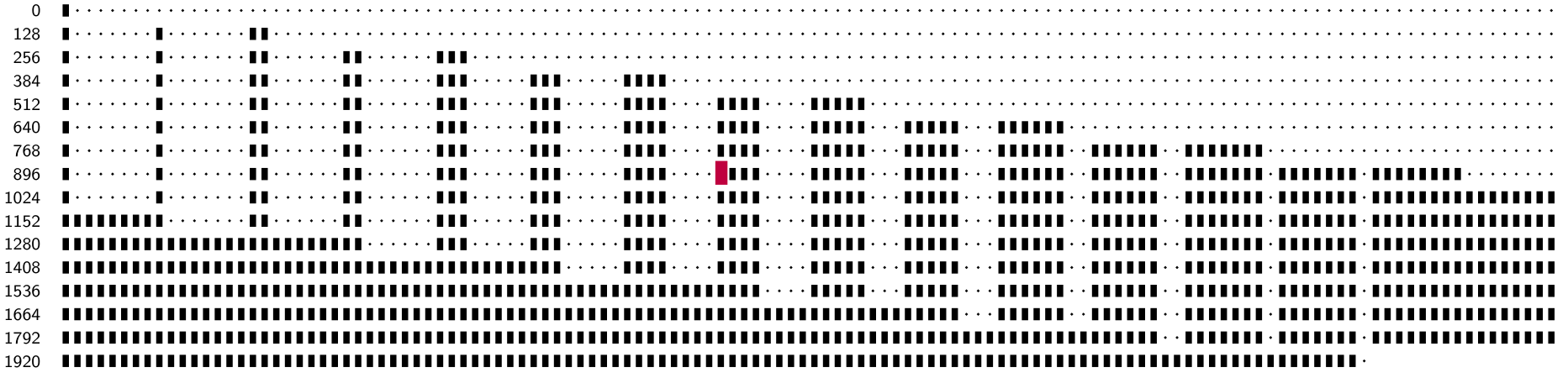}
    \caption{The Clifford defect of the Suzuki semigroup for $q_0 = 8$  is attained at $g = 952$, drawn in red.}
    \label{figure:suzuki}
\end{figure}

\begin{lemma}[Parametrization] \label{lemma:suzuki_parametrization}
The Suzuki semigroup $S$ can be expressed as
    \begin{equation*}
        S = \{\lambda_1 q + \lambda_2 q_0 + \lambda_3 \such 0 \leq 2\lambda_3 \leq \lambda_2 \leq 2\lambda_1, \quad \lambda_1, \lambda_2, \lambda_3 \in \mathbb{N} \}.
    \end{equation*}
\end{lemma}

\begin{proof}
    Let $R := \{\lambda_1 q + \lambda_2 q_0 + \lambda_3 \such 0 \leq 2\lambda_3 \leq \lambda_2 \leq 2\lambda_1, \quad \lambda_1, \lambda_2, \lambda_3 \in \mathbb{N}\}$. Let us see that $R \subseteq S$. To do so, we find some coefficients $a_1, a_2, a_3, a_4 \in \mathbb{N}$ successively such that $\lambda_1q + \lambda_2 q_0 + \lambda_3 \in R$ equals $a_1 q + a_2 (q + q_0) + a_3 (q + 2q_0) + a_4 (q + 2q_0 + 1) \in S$. First,
    \begin{equation*}
        \begin{split}
            \lambda_1 q + \lambda_2 q_0 + \lambda_3 &= (\lambda_1 - \lambda_3) q + (\lambda_2 - 2 \lambda_3) q_0 + \lambda_3 (q + 2q_0 + 1),
        \end{split}
    \end{equation*}
    and we set $a_4 := \lambda_3$. Second,
    \begin{equation*}
        \begin{split}
            (\lambda_1 - \lambda_3) q + (\lambda_2 - 2 \lambda_3) q_0 &= \left (\lambda_1 - \lambda_3 - \left \lfloor \frac{\lambda_2 - 2\lambda_3}{2} \right \rfloor \right) q + \left (\lambda_2 - 2 \lambda_3 - 2 \left \lfloor \frac{\lambda_2 - 2\lambda_3}{2} \right \rfloor \right) q_0 \\
            &+ \left \lfloor \frac{\lambda_2 - 2\lambda_3}{2} \right \rfloor (q + 2q_0)\\
            &= \left (\lambda_1 - \left \lfloor \frac{\lambda_2}{2} \right \rfloor \right) q + (\lambda_2 \mod 2) q_0 \\
            &+ \left (\left \lfloor \frac{\lambda_2}{2} \right \rfloor - \lambda_3\right )(q + 2q_0),\\
        \end{split}
    \end{equation*}
    and we set $a_3 := \left \lfloor \frac{\lambda_2}{2} \right \rfloor - \lambda_3$. Finally, we need to choose $a_3$ and $a_4$ from the remaining
    \begin{equation*}
        r :=  \left (\lambda_1 - \left \lfloor \frac{\lambda_2}{2} \right \rfloor \right) q + (\lambda_2 \mod 2) q_0.
    \end{equation*}
     If $(\lambda_2 \mod 2) = 0$, we are done by setting $a_2 := 0$ and $a_1 := \lambda_1 - \frac{\lambda_2}{2}$. If $(\lambda_2 \mod 2) = 1$, then
    \begin{equation*}
        r = \left (\lambda_1 - \frac{\lambda_2 - 1}{2} \right) q + q_0 = \left (\lambda_1 - \frac{\lambda_2 - 1}{2} - 1 \right) q + (q + q_0),
    \end{equation*}
    and we are done by setting $a_2 := 1$ and $a_1 := \lambda_1 - \frac{\lambda_2-1}{2} - 1$.

    For the converse, let $s := a_1 q + a_2 (q + q_0) + a_3 (q + 2q_0) + a_4 (q + 2q_0 + 1) \in S$ with $a_i \in \mathbb{N}$. Then
    \begin{equation*}
        s = (a_1 + a_2 + a_3 + a_4) q + (a_2 + 2a_3 + 2 a_4) q_0 + a_4.
    \end{equation*}
    Since $0 \leq 2a_4 \leq a_2 + 2a_3 + 2 a_4 \leq 2 (a_1 + a_2 + a_3 + a_4)$, we conclude that $s \in R$.
\end{proof}

\begin{lemma}[Order compatible] \label{lemma:suzuki_lex}
    Consider all  triples $\lambda := (\lambda_1, \lambda_2, \lambda_3)$ such that $0 \leq 2\lambda_3 \leq \lambda_2 \leq 2\lambda_1 \leq \frac{q-1}{q_0}$. Then
    \begin{equation*}
        \lambda_1 q + \lambda_2 q_0 + \lambda_3 \leq \lambda_1^\prime q + \lambda_2^\prime q_0 + \lambda_3^\prime \iff \lambda \preceq_\text{lex} \lambda^\prime.
    \end{equation*}
    In particular, each $s \in S$ with $s \leq g$ is uniquely represented by one  such $\lambda$.
\end{lemma}

\begin{proof}
    As in Lemma \ref{lemma:interval_parametrization}, it suffices to prove that $\lambda \llex \lambda^\prime$ implies that $\lambda_1 q + \lambda_2 q_0 + \lambda_3 < \lambda_1^\prime q + \lambda_2^\prime q_0 + \lambda_3^\prime$. If $\lambda_1 = \lambda_1^\prime$ and $\lambda_2 = \lambda_2^\prime$, then the assertion is obvious. If $\lambda_1 = \lambda_1^\prime$ and $\lambda_2 < \lambda_2^\prime$, we obtain $\lambda_3 - \lambda_3^\prime \leq q_0 - 1 < q_0 \leq (\lambda_2^\prime - \lambda_2)q_0$. If $\lambda_1 < \lambda_1^\prime$ then
    \begin{equation*}
        (\lambda_2 - \lambda_2^\prime) q_0 + \lambda_3 - \lambda_3^\prime \leq 2\lambda_1 q_0 + \lambda_3 \leq 2(\lambda_1^\prime - 1) q_0 + q_0 - 1 \leq q - q_0 - 2 < (\lambda_1^\prime - \lambda_1) q,
    \end{equation*}
    so the result follows. For the second part, if $s = \lambda_1 q + \lambda_2 q_0 + \lambda_3 \in S$ and $s \leq g = q_0 (q-1)$, then $\lambda_1 \leq q_0 - 1$ and so $2\lambda_1 \leq 2(q_0 -1) \leq \frac{q-1}{q_0}$. By then applying the first part to $s$, we are done.
\end{proof}

\begin{corollary} \label{corollary:suzuki}
    Let $S = \langle q, q + q_0, q + 2q_0, q + 2q_0 + 1 \rangle$ as before. Then $\sigma$ attains its maximum at $s := (q_0-1)(q + q_0)$ which is $\sigma(s) = \frac{q_0}{12}(4q - 3q_0 - 8)$.
\end{corollary}

\begin{proof}
    From \ref{theorem:ps_even}, we have that $\sigma$ is maximized at $s = g - \frac{q}{2} = q_0(q-1) - q_0^2 = (q_0-1)q + (q_0-1)q_0$.

    The remaining of the proof consists of computing $\sigma(s)$, which can be done thanks to Lemma \ref{lemma:suzuki_lex}. See the details in Appendix \ref{app:suzuki}.
\end{proof}

\section{Norm-trace semigroup} \label{section:norm-trace}

Given an integer $r \geq 2$, the norm-trace curve is a plane curve over $\mathbb{F}_q$ defined by $x^{\frac{q^r-1}{q-1}} = \sum_{i=0}^{r-1} y^{q^i}$. Observe that for $r=2$, this curve is the Hermitian curve \cite[Lemma 6.4.4]{stichtenoth}. It was introduced in \cite{geilNormTrace} for its applications to construct one-point codes at its only point at infinity. Its Weierstrass semigroup at this point is $$S := \langle q^{r-1}, \frac{q^r - 1}{q-1}\rangle.$$ It is symmetric (since it is generated by two elements), and its genus is $g(S) = \frac{q(q^{r-1} - 1)^2}{2(q-1)}$.

\begin{lemma}[Parametrization] \label{lemma:norm_trace_param}
The norm-trace semigroup $S$ can be expressed as
    \begin{equation*}
        S = \left \{\lambda_1 \frac{q^{r-1} - 1}{q-1} + \lambda_2 \in \mathbb{N} \such \frac{\lambda_1}{q} \leq \lambda_2 \leq \frac{\lambda_1}{q-1} \right \}.
    \end{equation*}
\end{lemma}

\begin{proof}
    Any element of $S$ can be written as
    \begin{equation*}
        a q^{r-1} + b \frac{q^r - 1}{q-1} = (a + b) q^{r-1} + b \frac{q^{r-1} - 1}{q-1} = ((a+b)(q-1) + b) \frac{q^{r-1} - 1}{q-1} + a + b
    \end{equation*}
    for some $a, b \in \mathbb{N}$. Clearly,
    \begin{equation*}
        \frac{(a + b) (q-1) + b}{q} \leq a + b \leq \frac{(a + b) (q-1) + b}{q-1},
    \end{equation*}
    from which we have $S \subseteq \left \{\lambda_1 \frac{q^{r-1} - 1}{q-1} + \lambda_2 \in \mathbb{N} \such \frac{\lambda_1}{q} \leq \lambda_2 \leq \frac{\lambda_1}{q-1} \right \}.$

    Now, let $\lambda_1 \frac{q^{r-1} - 1}{q-1} + \lambda_2$ with $\frac{\lambda_1}{q} \leq \lambda_2 \leq \frac{\lambda_1}{q-1}$. Since $\lambda_1 - (q-1) \lambda_2, q\lambda_2 - \lambda_1 \geq 0$, then
    \begin{equation*}
        \lambda_1 \frac{q^{r-1} - 1}{q-1} + \lambda_2 = (\lambda_1 - (q-1)\lambda_2) \frac{q^r - 1}{q-1} + (q\lambda_2 - \lambda_1) q^{r-1} \in S.
    \end{equation*}
\end{proof}

\begin{lemma}[Order compatible] \label{lemma:norm_trace_lex}
    Consider $\lambda_1, \lambda_2 \in \mathbb{N}$ and $\lambda_1^\prime, \lambda_2^\prime \in \mathbb{N}$ as in Lemma \ref{lemma:norm_trace_param}. If $\lambda_1 \frac{q^{r-1} - 1}{q-1} + \lambda_2, \lambda_1^\prime \frac{q^{r-1} - 1}{q-1} + \lambda_2^\prime \leq c(S)$, then
    \begin{equation*}
        \lambda_1 \frac{q^{r-1} - 1}{q-1} + \lambda_2 \leq \lambda_1^\prime \frac{q^{r-1} - 1}{q-1} + \lambda_2^\prime \iff (\lambda_1, \lambda_2) \leqlex (\lambda_1^\prime, \lambda_2^\prime).
    \end{equation*}
\end{lemma}

\begin{proof}
    It suffices to prove that the strict inequality holds if $(\lambda_1, \lambda_2) \llex (\lambda_1^\prime, \lambda_2^\prime)$. If $\lambda_1 = \lambda_1^\prime$, the result is trivial, so assume that $\lambda_1 < \lambda_1^\prime$. Since $c(S) = q(q^{r-1} - 1) \frac{q^{r-1} - 1}{q-1}$, we have that $\lambda_1, \lambda_1^\prime < q^r - 1$. Then
    \begin{equation*}
            \lambda_2 - \lambda_2^\prime \leq \frac{\lambda_1}{q-1} - \frac{\lambda_1^\prime}{q} \leq \frac{\lambda_1}{q-1} - \frac{\lambda_1 + 1}{q} = \frac{\lambda_1 - q + 1}{q(q-1)} < \frac{q^{r-1} - 1}{q-1} \leq (\lambda_1^\prime - \lambda_1) \frac{q^{r-1} - 1}{q-1},
    \end{equation*}
    so the result holds.
\end{proof}

\begin{proposition} \label{prop:norm-trace_dimension}
    Let $s = \lambda_1 \frac{q^{r-1} - 1}{q-1} + \lambda_2 \in S$ with $\lambda_1, \lambda_2 \in \mathbb{N}$ satisfying $\frac{\lambda_1}{q} \leq \lambda_2 \leq \frac{\lambda_1}{q-1}$. If $s \leq c(S)$, then
    \begin{equation*}
        l(s) = \sum_{i=0}^{\lambda_1 - 1} \left ( \left \lfloor \frac{i}{q-1} \right \rfloor - \left \lceil \frac{i}{q} \right \rceil \right) + \lambda_1 + \lambda_2 - \left \lceil \frac{\lambda_1}{q} \right \rceil.
    \end{equation*}
\end{proposition}

\begin{proof}
    The result follows from Lemmas \ref{lemma:norm_trace_param} and  \ref{lemma:norm_trace_lex} by a simple counting argument.
\end{proof}

From Lemma \ref{lemma:norm_trace_lex} and Remark \ref{remark:s_minus_1}, we have that $\sigma$ is maximized at some $\lambda \frac{q^{r-1} - 1}{q-1} + \left \lceil \frac{\lambda}{q} \right \rceil$. Motivated by this, for the rest of the section we simply write $\sigma(\lambda)$ instead of $\sigma\left (\lambda \frac{q^{r-1} - 1}{q-1} + \left \lceil \frac{\lambda}{q} \right \rceil \right)$.

\begin{lemma} \label{lemma:norm-trace_sigma}
    Let $\lambda = kq + e$ with $0 \leq e \leq q-1$. Under the hypothesis of Lemma \ref{lemma:norm_trace_lex}, $\sigma(\lambda) \leq \sigma(\lambda+1)$ if and only if $\frac{q^{r-1} -1}{q-1} - 2\left \lfloor \frac{k+e}{q-1} \right \rfloor + \left \lceil \frac{e}{q} \right \rceil - 1 \geq 0$.
\end{lemma}

\begin{proof}
    From Proposition \ref{prop:norm-trace_dimension}, we have
    \begin{equation*}
    \begin{split}
        2(\sigma(\lambda + 1) - \sigma(\lambda)) &= \frac{q^{r-1} - 1}{q-1} + \left \lceil \frac{\lambda + 1}{q}\right \rceil + \left \lceil \frac{\lambda}{q} \right \rceil - 2\left \lfloor \frac{\lambda}{q-1} \right \rfloor - 2\\
        &= \frac{q^{r-1} - 1}{q-1} + k + \left \lceil \frac{e + 1}{q}\right \rceil + k + \left \lceil \frac{e}{q} \right \rceil -2k - 2\left \lfloor \frac{k+e}{q-1} \right \rfloor - 2\\
        &= \frac{q^{r-1} - 1}{q-1} + \left \lceil \frac{e}{q} \right \rceil - 2\left \lfloor \frac{k+e}{q-1} \right \rfloor - 1,
    \end{split}
    \end{equation*}
    and the result follows.
\end{proof}

\begin{lemma}
    Let $q$ and $r$ be such that $\frac{q^{r-1} -1}{2(q-1)} \notin \mathbb{N}$. Let $\lambda = kq + e$ with $0 \leq e \leq q-1$ and $k = \frac{q^{r-1} -1}{2} - \frac{q-1}{2} + t$ for $0 \leq t \leq q-2$. Then, for a fixed $k$, $\sigma$ attains its maximum when $e = q-1-t$. Moreover, if $q>2$ this maximum is unique, and if $q=2$, $\sigma$ is constant.
\end{lemma}

\begin{proof}
    First observe that
    $
        \frac{k+e}{q-1} = \frac{q^{r-1} - 1}{2(q-1)} - \frac{1}{2} + \frac{e+t}{q-1}$ implies
    \begin{equation*} \left \lfloor \frac{k+e}{q-1} \right \rfloor = \frac{q^{r-1} - 1}{2(q-1)} - \frac{1}{2} + \left \lfloor \frac{e+t}{q-1} \right \rfloor = \frac{q^{r-1} - q}{2(q-1)} + \left \lfloor \frac{e+t}{q-1} \right \rfloor.
    \end{equation*}
    Let us use Lemma \ref{lemma:norm-trace_sigma}. We have
    \begin{equation*}
        M := \frac{q^{r-1} -1}{q-1} - 2\left \lfloor \frac{k+e}{q-1} \right \rfloor + \left \lceil \frac{e}{q} \right \rceil - 1 = \left \lceil \frac{e}{q} \right \rceil - 2\left \lfloor \frac{e+t}{q-1} \right \rfloor.
    \end{equation*}
    We distinguish the following cases:
    \begin{enumerate}
        \item If $e=0$, then $M = 0$.
        \item If $1 \leq e \leq q-2-t$, then $\left \lfloor \frac{e+t}{q-1} \right \rfloor \leq \left \lfloor \frac{q-2}{q-1} \right \rfloor = 0$ and so $M \geq 1$.
        \item If $q-1-t \leq e \leq q-1$, then $\left \lfloor \frac{e+t}{q-1} \right \rfloor \geq \left \lfloor \frac{q-1}{q-1} \right \rfloor = 1$ and so $M \leq -1$.
    \end{enumerate}
    Hence, $\sigma(kq + e)$ as a function depending on $e$ is strictly increasing for $1 \leq e \leq q-1-t$ and strictly decreasing for $q-1-t \leq e \leq q-1$. For $e=0$, we have $\sigma(kq) = \sigma(kq + 1)$. The rest follows easily.
\end{proof}

\begin{theorem}
    Let $S$ be the semigroup of the norm-trace curve. We then have the following results.
    \begin{enumerate}
        \item If $q$ is odd, $\sigma$ is maximized at $g$ and
            \begin{equation*}
                \sigma(g) =
                \begin{cases}
                    \frac{1}{8(q-1)} (q^{2r-1} - 4q^r + 2q^{r-1} + q^2 + q - 1) &\text{if } r \text{ is even}\\
                    \frac{1}{8(q-1)} (q^{2r - 1} + q^{2r - 4} - q^{2r - 6} - 4q^r + 3q^{r-3} + q^2 + 3q - 4) &\text{if } r \text{ is odd}.
                \end{cases}
            \end{equation*}
        \item If $q$ is even, $\sigma$ is maximized at $g - \frac{q^{r-1}}{2}$ and
            \begin{equation*}
                \sigma\left (g - \frac{q^{r-1}}{2} \right) = \frac{q}{8(q-1)} \left ( q^{2r - 2} - 4 q^{r-1} + 2q^{r-2} + q \right).
            \end{equation*}
    \end{enumerate}
\end{theorem}

\begin{proof}
    First, we study where $\sigma(\lambda)$ is maximized. Write $\lambda = q k + e$ with $0 \leq k \leq q-1$.
    \begin{enumerate}
        \item If $q$ is odd, then we write
            \begin{equation*}
                g = \frac{q^{r-1} - 1}{q-1}\left (q \left (\frac{q^{r-1} - 1}{q-1} - 1 \right )  + \frac{q+1}{2} \right) + \frac{q^{r-1} - 1}{2} \in S.
            \end{equation*}
            Because of Lemma \ref{lemma:norm_trace_lex} and Theorem \ref{theorem:symmetric_bound}, we have $k \leq \frac{q^{r-1}-1}{2} - 1$ and so $$\frac{k+e}{q-1} \leq \frac{q^{r-1} - 1}{2(q-1)} + \frac{e-1}{q-1}.$$
            We distinguish cases depending if $\frac{q^{r-1} - 1}{2(q-1)} \in \mathbb{N}$ or not, or equivalently, if $r$ is odd or even, respectively.
            \begin{enumerate}
                \item If $\frac{q^{r-1} - 1}{2(q-1)} \in \mathbb{N}$, then $\left \lfloor \frac{k+e}{q-1} \right \rfloor \leq \frac{q^{r-1} - 1}{2(q-1)} + \left \lfloor \frac{e-1}{q-1} \right \rfloor$ and
                    \begin{equation*}
                        \frac{q^{r-1} - 1}{q-1} - 2 \left \lfloor \frac{k+e}{q-1} \right \rfloor + \left \lceil \frac{e}{q} \right \rceil - 1 \geq \left \lceil \frac{e}{q} \right \rceil - 2 \left \lfloor \frac{e-1}{q-1} \right \rfloor - 1 =
                        \begin{cases}
                            1 & \text{if } e= 0\\
                            0 & \text{if } e\neq 0.
                        \end{cases}
                    \end{equation*}
                    Lemma \ref{lemma:norm-trace_sigma} gives us that $\sigma$ is increasing in $[0,g]$.

                \item Suppose that $\frac{q^{r-1} - 1}{2(q-1)} \notin \mathbb{N}$. Applying Theorem \ref{theorem:symmetric_bound}, we get that $\sigma$ is maximized at some point of $[g - \frac{q^{r-1+1}}{2}, g]$. Since
                    \begin{equation*}
                    \begin{split}
                        g - \frac{q^{r-1} +1}{2} &= \frac{q^{r-1}-1}{q-1} \left (q\left (\frac{q^{r-1}-1}{2} - 1\right ) +1 \right) + \frac{q^{r-1} -1}{2} - 1 \\
                        &\geq \frac{q^{r-1}-1}{q-1} \left (q\left (\frac{q^{r-1}-1}{2} - 1\right )\right) + \frac{q^{r-1} -1}{2} - 1 \in S,
                        \end{split}
                    \end{equation*}
                    we obtain from Lemma \ref{lemma:norm_trace_lex} that $k=\frac{q^{r-1}-1}{2} - 1$. From Lemma \ref{lemma:norm-trace_sigma}, $\sigma$ is maximized when $e = q-1-t = q-1-\frac{q-1}{2}+1 = \frac{q+1}{2}$, that is, at $g$.
            \end{enumerate}

        \item If $q$ is even, then we can write
            \begin{equation*}
                g - \frac{q^{r-1}}{2} = \frac{q^{r-1}-1}{q-1} \left (q\left(\frac{q^{r-1}}{2}-2 \right) + \frac{q}{2} + 1 \right) + \frac{q^{r-1}}{2} - 1.
            \end{equation*}
            From now on, suppose that $q > 2$. Then, the last expression is written as in Lemma \ref{lemma:norm_trace_param}, and so $g - \frac{q^{r-1}}{2} \in S$. Lemma \ref{lemma:norm-trace_sigma} gives us that $\sigma(\lambda)$ is maximized at $g - \frac{q^{r-1}}{2}$ among the elements having $k = \frac{q^{r-1}}{2} - 2$. If we fix $k = \frac{q^{r-1}}{2}-1$, then $\sigma(\lambda)$ is maximized when $e = \frac{q}{2}$. Since
            \begin{equation*}
                g - 1 = \frac{q^{r-1}-1}{q-1} \left (q\left(\frac{q^{r-1}}{2}-1 \right) \right) + \frac{q(q^{r-1}-1)}{2(q-1)} - 1 < \frac{q^{r-1}-1}{q-1} (kq + e) + k+1,
            \end{equation*}
            $k=\frac{q^{r-1}}{2}$, which follows from Lemmas \ref{lemma:norm-trace_sigma} and \ref{lemma:norm_trace_lex}. If $q=2$, then
            \begin{equation*}
                g - \frac{q^{r-1}}{2} = \frac{q^{r-1}-1}{q-1} \left (q\left(\frac{q^{r-1}}{2}-1 \right) \right) + \frac{q^{r-1}}{2} - 1.
            \end{equation*}
            and the result follows from Lemma \ref{lemma:norm-trace_sigma}.
    \end{enumerate}
    The explicit formula for the maximum of $\sigma$ is given in Appendix \ref{app:norm-trace}.
\end{proof}

\section{Conclusion}
In this paper, we introduced the Clifford defect of a numerical semigroup, taking as inspiration Weierstrass semigroups of points on curves over finite fields. We considered some of its properties. In addition, we determined its value or where it is achieved for several families of semigroups. We leave as an open question to determine the Clifford defect of other families of numerical semigroups, including those with exactly two generators.

\section*{Acknowledgements}

E. Camps is partially supported by NSF DMS-2201075 and the Commonwealth Cyber Initiative.
G. L. Matthews is partially supported by NSF DMS-2502705 and the Commonwealth Cyber Initiative.

A. Fidalgo-D{\'i}az and U. Mart{\'i}nez-Pe{\~n}as are partially supported by Grant PID2022-138906NB-C21 funded by MICIU/AEI/10.13039/501100011033 and by ERDF/EU.

\nocite{*}
\printbibliography

\begin{appendix}

\section{Computations for Suzuki} \label{app:suzuki}

In Corollary \ref{corollary:suzuki} was shown that the Clifford defect of the Suzuki curve it is attained at $s := (q_0-1)q + (q_0-1)q_0$. Let us compute an explicit value for $\sigma(s)$ using lemmas \ref{lemma:suzuki_parametrization} and \ref{lemma:suzuki_lex}. First,
\begin{equation*}
    l(s) = \sum_{\lambda_1 = 0}^{q_0 -2} \sum_{\lambda_2=0}^{2\lambda_1} \left (\left \lfloor \frac{\lambda_2}{2} \right \rfloor + 1 \right ) + \sum_{\lambda_2 = 0}^{q_0-2} \left ( \left \lfloor \frac{\lambda_2}{2} \right \rfloor + 1 \right) + 1.
\end{equation*}
We expand the first summatory:
\begin{equation*}
    \begin{split}
        \sum_{\lambda_1 = 0}^{q_0 -2} \sum_{\lambda_2=0}^{2\lambda_1} \left (\left \lfloor \frac{\lambda_2}{2} \right \rfloor + 1 \right ) &= \sum_{\lambda_1 = 0}^{q_0 - 2} \left( 2 \lambda_1 + 1 + \sum_{\lambda_2 = 0}^{2\lambda_1} \left \lfloor \frac{\lambda_2}{2} \right \rfloor\right)\\
        &= \sum_{\lambda_1 = 0}^{q_0 - 2} \left( 2 \lambda_1 + 1 + \lambda_1 + 2\sum_{\lambda_2 = 0}^{\lambda_1-1} \lambda_2\right)\\
        &= \sum_{\lambda_1 = 0}^{q_0 - 2} \left( \lambda_1^2 + 2 \lambda_1 + 1 \right)\\
        %&= \sum_{\lambda_1 = 0}^{q_0 - 2} ( \lambda_1 + 1 )^2\\
        &= \sum_{\lambda_1 = 1}^{q_0 - 1} ( \lambda_1 + 1 )^2\\
        %&= \frac{(q_0 - 1) q_0 (2q_0 -1)}{6}\\
        &= \frac{q_0}{6}(2q_0^2 - 3q_0 + 1).
    \end{split}
\end{equation*}
We expand the second summatory:
\begin{equation*}
    \begin{split}
        \sum_{\lambda_2 = 0}^{q_0-2} \left ( \left \lfloor \frac{\lambda_2}{2} \right \rfloor + 1 \right) + 1 &= q_0 - 1 + \sum_{\lambda_2 = 0}^{q_0-2} \left \lfloor \frac{\lambda_2}{2} \right \rfloor \\
        &= q_0 - 1 + \frac{q_0}{2} - 1 + 2 \sum_{\lambda_2=0}^{\frac{q_0}{2} - 2} \lambda_2\\
        &= \frac{3q_0}{2} - 2 + \left (\frac{q_0}{2} - 1 \right) \left ( \frac{q_0}{2} - 2\right )\\
        &= \frac{q_0^2}{4}.
    \end{split}
\end{equation*}
Adding both:
\begin{equation*}
    l(s) = \frac{q_0}{12} (4 q_0^2 - 3q_0 + 2) + 1.
\end{equation*}
From this point it is straightforward to see that $\sigma(s) = \frac{q_0}{12}(8q_0^2 - 3q_0 - 8)$.

\section{Computations for the Norm-trace} \label{app:norm-trace}

We have proof that the Clifford is attained at some element of the form $s := \lambda \frac{q^{r-1} - 1}{q-1} + \left \lceil \frac{\lambda}{q} \right \rceil$, with $\lambda$ depending on the parity of $q$. Let us move towards computing $\ell(s)$ using Proposition \ref{prop:norm-trace_dimension}.

If we write $\lambda = k q + e$, then
\begin{equation*}
    \sum_{i=0}^{\lambda - 1} \left \lfloor \frac{i}{q-1} \right \rfloor - \left \lceil \frac{i}{q-1} \right \rceil = \binom{k+2}{2} - \lambda + \sum_{i= 0}^{k+e -1} \left \lfloor \frac{i}{q-1} \right \rfloor.
\end{equation*}
This follows from
\begin{equation*}
    \begin{split}
        \sum_{i=0}^{\lambda - 1} \left \lfloor \frac{i}{q-1} \right \rfloor &= \sum_{i=0}^{k(q-1) - 1} \left \lfloor \frac{i}{q-1} \right \rfloor + \sum_{i=k(q-1)}^{\lambda - 1} \left \lfloor \frac{i}{q-1} \right \rfloor \\
        &= (q-1) \sum_{i=0}^{k-1} i + \sum_{i=0}^{k+e-1} \left \lfloor \frac{i + k(q-1)}{q-1} \right \rfloor\\
        &= (q-1) \binom{k}{2} + (k+e) k + \sum_{i=0}^{k+e-1} \left \lfloor \frac{i}{q-1} \right \rfloor,
    \end{split}
\end{equation*}
together with
\begin{equation*}
    \begin{split}
        \sum_{i=0}^{\lambda - 1} \left \lceil \frac{i}{q} \right \rceil &= \sum_{i=0}^{kq} \left \lceil \frac{i}{q} \right \rceil + \sum_{i=kq+1}^{\lambda-1} \left \lceil \frac{i}{q} \right \rceil\\
        &= q \sum_{i=0}^{k} i + \sum_{i=1}^{e-1} \left \lceil \frac{i + kq}{q} \right \rceil\\
        &= q \binom{k+1}{2} + k(e-1) + \sum_{i=0}^{e-1} \left \lceil \frac{i}{q} \right \rceil\\
        &= q \binom{k+1}{2} + k(e-1) + e-1.
    \end{split}
\end{equation*}
Then
\begin{equation*}
    \begin{split}
        \sum_{i=0}^{\lambda - 1} \left \lfloor \frac{i}{q-1} \right \rfloor - \left \lceil \frac{i}{q-1} \right \rceil &=
        q \left [ \binom{k}{2} - \binom{k+1}{2} \right ] - \binom{k}{2} + k^2 + k - e + 1 + \sum_{i=0}^{k+e-1} \left \lfloor \frac{i}{q-1} \right \rfloor\\
        &= - qk - \binom{k}{2} + 2 \binom{k+1}{2} - e + 1 + \sum_{i=0}^{k+e-1} \left \lfloor \frac{i}{q-1} \right \rfloor\\
        &= - \lambda + k + \binom{k+1}{2} + 1 + \sum_{i=0}^{k+e-1} \left \lfloor \frac{i}{q-1} \right \rfloor\\
        &= - \lambda + \binom{k+2}{2} + \sum_{i=0}^{k+e-1} \left \lfloor \frac{i}{q-1} \right \rfloor.
     \end{split}
\end{equation*}

Moreover, if we write $k+e = (q-1)k^\prime + e^\prime$ with $0 \leq e^\prime < q-1$, then
\begin{equation*}
    \begin{split}
        \sum_{i= 0}^{k+e -1} \left \lfloor \frac{i}{q-1} \right \rfloor &= \sum_{i= 0}^{k^\prime(q-1) -1} \left \lfloor \frac{i}{q-1} \right \rfloor + \sum_{i=k^\prime (q-1)}^{k+e-1} \left \lfloor \frac{i}{q-1} \right \rfloor\\
        &= (q-1) \sum_{i= 0}^{k^\prime -1} i + \sum_{i=0}^{e^\prime-1} \left \lfloor \frac{i + k^\prime (q-1)}{q-1} \right \rfloor\\
        &= (q-1) \binom{k^\prime}{2} + e^\prime k^\prime + \sum_{i=0}^{e^\prime-1} \left \lfloor \frac{i}{q-1} \right \rfloor\\
        &= (q-1) \binom{k^\prime}{2} + e^\prime k^\prime .
    \end{split}
\end{equation*}

So with this notation and Proposition \ref{prop:norm-trace_dimension},
\begin{equation*}
    \begin{split}
        l(s) &= \left (\sum_{i=0}^{\lambda - 1} \left \lfloor \frac{i}{q-1} \right \rfloor - \left \lceil \frac{i}{q-1} \right \rceil \right) + \lambda + 1\\
        &= 1 + \binom{k+2}{2} + \sum_{i=0}^{k+e-1} \left \lfloor \frac{i}{q-1} \right \rfloor\\
        &= 1 + \binom{k+2}{2} + (q-1) \binom{k^\prime}{2} + e^\prime k^\prime .
    \end{split}
\end{equation*}
At this point, we need to distinguish cases for computing the explicit value of $l(s)$.

\subsection{$q$ even}

Recall that the Clifford defect is attained when $\lambda = q \left (\frac{q^r}{2} - 2 \right ) + \frac{q}{2} + 1$. Set $k = \frac{q^r}{2} - 2$ and $e = \frac{q}{2} + 1$. Then $k + e = (q-1) \left (\frac{q (q^{r-2} - 1)}{2(q-1)} + 1 \right )$. Set $k^\prime = \frac{q (q^{r-2} - 1)}{2(q-1)} + 1$ and $e^\prime = 0$. Unwrapping the notation, it is a straighforward computation to see that
\begin{equation*}
    l(s) = 1 + \binom{\frac{q^r}{2}}{2} + (q-1) \binom{\frac{q (q^{r-2} - 1)}{2(q-1)} + 1}{2} = 1 + \frac{q}{8(q-1)} \left ( q^{2r-2} - 2q^{r-1} - q + 2\right ).
\end{equation*}

And finally,
\begin{equation*}
    \sigma(s) = \frac{g}{2} - \frac{q^{r-1}}{4} - l(s) + 1 = \frac{q}{8(q-1)} \left ( q^{2r - 2} - 4 q^{r-1} + 2q^{r-2} + q \right).
\end{equation*}

\subsection{$q$ odd}

The Clifford defect is attained at $g$, that is, $\lambda = \frac{q^{r-1} - 3}{2} + \frac{q+1}{2}$. Set $k = \frac{q^{r-1} - 1}{2} - 1$ and $e = \frac{q+1}{2}$. Then $k + e = (q-1)\frac{q^{r-2} + 1}{2} + \frac{q^{r-2} - 1}{2}$. We distinguish cases depending on the parity of $r$:
\begin{enumerate}
    \item If $r$ is even, then $q - 1$ divides $\frac{q^{r-2} - 1}{2}$. We set $k^\prime = \frac{q^{r-2} + 1}{2} + \frac{q^{r-2} - 1}{2(q-1)}$ and $e^\prime = 0$. Then we are able to compute
        \begin{equation*}
            l(g) = 1 + \frac{1}{8(q-1)} \left ( q^{2r - 1} - 2q^{r-1} - q^2 + q + 1\right ).
        \end{equation*}
        and the Clifford defect is
        \begin{equation*}
            \sigma(g) = \frac{1}{8(q-1)} (q^{2r-1} - 4q^r + 2q^{r-1} + q^2 + q - 1).
        \end{equation*}

    \item If $r$ is odd, then $\frac{q^{r-1} -1}{2} = (q-1) \left (\frac{q^{r-2} - 1}{2(q-1)} - \frac{1}{2} \right ) + \frac{q-1}{2}$, being $\frac{q^{r-2} - 1}{2(q-1)} - \frac{1}{2}$ a positive integer. We set $k^\prime = \frac{q^{r-2} + 1}{2} + \frac{q^{r-2} - 1}{2(q-1)} - \frac{1}{2}$ and $e^\prime = \frac{q-1}{2}$. After a simple but long computation shows that
        \begin{equation*}
            l(g) = 1 + \frac{1}{8(q-1)} (q^{2r - 1} - 3q^{r-1} - q + 3)
        \end{equation*}
        and the Clifford defect is
        \begin{equation*}
            \sigma(g) = \frac{1}{8(q-1)} (q^{2r - 1} + q^{2r - 4} - q^{2r - 6} - 4q^r + 3q^{r-3} + q^2 + 3q - 4).
        \end{equation*}
\end{enumerate}
\end{appendix}

\end{document}